\documentclass[10pt,oneside]{amsart}
\usepackage[margin=1in,letterpaper]{geometry}  
\usepackage{amssymb,amsthm}
\usepackage[leqno]{amsmath}
\usepackage{enumerate}             
\usepackage{srcltx}                
\usepackage{cite}
\usepackage{hyperref}
\usepackage{xspace}                
\usepackage{ctable}                
\usepackage{multirow}              
\usepackage{caption}
\usepackage{picins}
\usepackage{cutwin}
\usepackage[percent]{overpic}
\usepackage{mathrsfs}
\usepackage{esvect}

\newbox\mybox
\newlength{\myboxwidth}

\newcommand\addpicture[5]{%
  \setbox\mybox=\hbox{\includegraphics[width=#3]{#2}}
  \myboxwidth\wd\mybox    
  \renewcommand\windowpagestuff{%
    \includegraphics[width=#3]{#2}
    \captionsetup{hypcap=false}
    \captionof{figure}{#4}
    \label{#5}
  }
  \parpic[#1]{%
    \begin{minipage}{\myboxwidth}
      \windowpagestuff 
    \end{minipage} 
  }
}

\newcommand\addpicturenocap[4]{%
  \setbox\mybox=\hbox{\includegraphics[width=#3]{#2}}
  \myboxwidth\wd\mybox    
  \renewcommand\windowpagestuff{%
    \includegraphics[width=#3]{#2}
  }
  \parpic[#1]{%
    \begin{minipage}{\myboxwidth}
      \windowpagestuff 
    \end{minipage} 
  }
}

\makeatletter
\def\@settitle{\begin{center}%
  \baselineskip14\p@\relax
    \normalfont \bf
\uppercasenonmath\@title
  \@title
\end{center}%
\begin{center}%
{\it A Picturebook of Tilings}
\end{center}%
}
\makeatother

\DeclareSymbolFont{bbold}{U}{bbold}{m}{n}
\DeclareSymbolFontAlphabet{\mathbbold}{bbold}

\newtheorem{theorem}{Theorem}[section]
\newtheorem*{theorem-nn}{Theorem}
\newtheorem{corollary}[theorem]{Corollary}
\newtheorem{lemma}[theorem]{Lemma}
\newtheorem{proposition}[theorem]{Proposition}

\newtheorem*{question-nn}{Question}

\newtheorem*{conjecture-nn}{Conjecture}
\theoremstyle{definition}
\newtheorem{definition}[theorem]{Definition}
\newtheorem*{definition-nn}{Definition}
\theoremstyle{remark}

\newtheorem*{example-nn}{Example}

\newcommand{\acts}{\curvearrowright}
\newcommand{\es}{\varnothing}
\newcommand{\orbit}{\mathrm{Orb}}
\newcommand{\one}{\skew{-3}{\vec}{\mathbf{1}}}
\newcommand{\la}{\leftarrow}

\begin{document}
\title{Lebesgue Orbit Equivalence of Multidimensional Borel Flows}

\author{Konstantin Slutsky}
\address{Department of Mathematics, Statistics, and Computer Science\\
University of Illinois at Chicago\\
322 Science and Engineering Offices (M/C 249)\\
851 S. Morgan Street\\
Chicago, IL 60607-7045}
\email{kslutsky@gmail.com}

\begin{abstract}
  The main result of the paper is classification of free multidimensional Borel flows up to Lebesgue
  Orbit Equivalence, by which we understand an orbit equivalence that preserves the
  Lebesgue measure on each orbit.  Two non smooth \( \mathbb{R}^{d} \)-flows are shown to be
  Lebesgue Orbit Equivalence if and only if they admit the same number of invariant ergodic
  probability measures.
\end{abstract}

\maketitle

\picchangemode
\section{Introduction}
\label{sec:introduction}

\subsection{Orbit equivalences}
\label{sec:orbit-equivalences}

The core concept for this paper is orbit equivalence of actions.  Since we work in the framework of
Descriptive Set Theory, we consider Borel actions of Polish groups\footnote{A topological group is
  \emph{Polish} if its topology is Polish, i.e., it is separable and metrizable by a complete
  metric.} on standard\footnote{A standard Borel space is a set \( X \) together with a
  distinguished \( \sigma \)-algebra \( \Sigma \) such that for some Polish topology on \( X \) the
  \( \sigma \)-algebra~\( \Sigma \) coincides with the family of Borel sets.} Borel spaces.  Two
group actions \( G \acts X \) and \( H \acts Y \) are \emph{orbit equivalent} (OE) if there exists a
Borel bijection \( \phi : X \to Y \) which sends orbits onto orbits:
\[ \phi \bigl( \orbit_{G}(x) \bigr)  = \orbit_{H}\bigl( \phi(x) \bigr) \quad \textrm{for
all } x \in X. \]

This notion is, perhaps, better known in the framework of Ergodic Theory, where phase spaces \( X \)
and \( Y \) are endowed with invariant probability measures which the orbit equivalence is assumed
to preserve.  Our set up is different in two aspects.  We do not fix any measures on the phase
spaces, therefore potentially increasing the choice of orbit equivalence maps.  On the other hand,
contrary to the Ergodic Theoretical case where functions need to be defined almost everywhere, we
require our maps \( \phi : X \to Y \) to be defined on each and every point.  In this sense
Descriptive Theoretical set up is more restrictive.

The notion of OE in the above form is arguably better suited for actions of \emph{discrete} groups.
One of the high points in that area is the classification of hyperfinite equivalence relations up to
orbit equivalence by R.~Dougherty, S.~Jackson, and A.~S.~Kechris \cite[Theorem
9.1]{dougherty_structure_1994} based on an earlier work of M.~G.~Nadkarni
\cite{nadkarni_existence_1990}.  First, let us recall that with any action \( G \acts X \) we may
associate the \emph{orbit equivalence relation} on~\( X \) (denoted by \( E_{X}^{G} \), or just
by \( E_{X} \), when the group is understood) by declaring two points to be equivalent whenever
they are in the same orbit of the action.  A countable Borel equivalence relation is
\emph{hyperfinite} if it is an increasing union of finite Borel equivalence relations.  Given a
countable Borel equivalence relation \( E \subseteq X \times X \), a measure \( \mu \) on \( X \) is
said to be \( E \)-invariant if every Borel automorphism \( \theta : X \to X \) of \( E \) preserves
\( \mu \).  A measure is \emph{ergodic} with respect to \( E \) if every \( E \)-invariant subset of
\( X \) is either null or co-null.  For the sake of brevity, a probability invariant ergodic measure
is called a \emph{pie} measure.  A countable relation is \emph{aperiodic} if all of its classes are
infinite.  Finally, recall that a countable equivalence relation is \emph{smooth} if it admits a
Borel transversal \textemdash{} a Borel set that intersects every equivalence class in exactly one
point.  Having introduced all the necessary terminology, the Dougherty--Jackson--Kechris
classification (we later refer to it as DJK classification for short) of hyperfinite Borel
equivalence relations can be stated as follows.

\begin{theorem-nn}[Dougherty--Jackson--Kechris]
  \label{thm:DJK}
  Two non smooth aperiodic hyperfinite Borel equivalence relations are isomorphic if and only if
  cardinalities of the sets of pie measures are the same.
\end{theorem-nn}

The situation changes drastically when one considers locally compact non discrete groups.  All free
non smooth Borel \( \mathbb{R} \)-flows are orbit equivalent.  In fact, a much stronger result is
true.  An orbit equivalence \( \phi : X \to Y \) between two \emph{free} actions
\( \mathbb{R} \acts X \) and \( \mathbb{R} \acts Y \) gives rise to a cocycle\footnote{For abelian
  groups the action is denoted additively.  For instance, \( x + r \) below means the action of
  \( r \in \mathbb{R} \) upon the element \( x \in X \).}
\[ f : \mathbb{R} \times X \to \mathbb{R} \qquad f(r,x) \textrm{ is such that } \phi(x) + f(r,x) =
\phi(x + r). \]
A \emph{time-change equivalence} between free actions \( \mathbb{R} \acts X \) and
\( \mathbb{R} \acts Y \) is an OE \( \phi : X \to Y \) such that for each \( x \in X \) the function
\( f(\,\cdot\,, x) : \mathbb{R} \to \mathbb{R} \) is a homeomorphism\footnote{One may even require a
  stronger condition of \( f(\cdot, x) \) being a \( C^{\infty} \)-diffeomorphism.  The theorem of
  Miller and Rosendal remains valid in this case.}\kern-1.2mm.  This is a
substantial strengthening of the notion of orbit equivalence.  Nevertheless, as proved by
B.~D.~Miller and C.~Rosendal \cite{miller_descriptive_2010}, in the Descriptive Set Theoretic set up
the world of free \( \mathbb{R} \)-flows collapses with respect to time-change equivalence.

\begin{theorem-nn}[Miller--Rosendal]
  \label{thm:Miller-Rosendal}
  Any two non smooth free \( \mathbb{R} \)-flows are time-change equivalent.
\end{theorem-nn}

The difference between continuous and discrete worlds lies in the fact that continuous groups have a
lot more non trivial structures on them.  The obvious one is topology.  Whenever we have a free
action \( G \acts X \) and an orbit \( \mathcal{O} \subseteq X \), we may transfer the topology
from \( G \) onto \( \mathcal{O} \) using the correspondence \( G \ni g \mapsto gx \in \mathcal{O}
\) for any chosen \( x \in \mathcal{O} \).  If groups \( G \) and \( H \) are discrete, any OE between
their free actions respects the topology: restricted onto any orbit \( \mathcal{O} \subseteq X \),
\( \phi : \mathcal{O} \to \phi(\mathcal{O}) \) is a homeomorphism.  When the topology on
\( \mathcal{O} \) is not discrete, the map \( \phi : \mathcal{O} \to \phi(\mathcal{O}) \) has no
reasons to preserve the topology, and when this is imposed as an additional assumption on
\( \phi \), one recovers the concept of time-change equivalence.

The structure possessed by all locally compact groups which is responsible for the failure of DJK
classification is Haar measure.  Being invariant, it can also be transferred\footnote{Since we
  consider \emph{left} actions one needs to take a \emph{right} Haar measure.  More on this in
  Section \ref{sec:orbit-equiv-locally}.} onto any orbit of a free action \( G \acts X \).  Again,
if \( G \) and \( H \) are discrete (and if one takes counting Haar measures), any OE map
\( \phi : X \to Y \) restricts to a measure preserving isomorphism between orbits.  When \( G \) and
\( H \) are continuous, this may no longer be the case.  This turns out to be an obstacle for
cardinality of the set of pie measures to be an invariant of the OE between non discrete locally
compact group actions.

\subsection{Lebesgue Orbit Equivalence}
\label{sec:LOE-intro}

The paper is concerned mainly with free actions of Euclidean spaces \( \mathbb{R}^{d} \acts X \) on
standard Borel spaces.  Two such actions \( \mathbb{R}^{d} \acts X \) and
\( \mathbb{R}^{d} \acts Y \) are \emph{Lebesgue Orbit Equivalent} (LOE) if there exists an OE
\( \phi : X \to Y \) which preserves the Lebesgue measure on each orbit\footnote{A rigorous
  definition can be found in Section \ref{sec:orbit-equiv-locally}.}\kern-1.2mm.  In Ergodic
Theoretical set up, i.e., when \( X \) and \( Y \) are endowed with probability invariant measures
and the map \( \phi \) needs to be defined almost everywhere, this notion seems to have appeared for
the first time in the work of U.~Krengel \cite{krengel_rudolphs_1976}, where the following theorem
was proved.

\begin{theorem-nn}[Krengel]
  \label{thm:LOE-ErgTh-Krengel}
  Free ergodic flows \( \mathbb{R} \acts X \) and \( \mathbb{R} \acts Y \) are always LOE.
\end{theorem-nn}

Still withing the framework of Ergodic Theory, this has later been generalized by D.~Rudolph
\cite{rudolph_smooth_1979} to free actions of \( \mathbb{R}^{d} \).  In fact, Rudolph proved a much
stronger result. Namely for \( d \ge 2 \) an OE map \( \phi : X \to Y \) may be assumed to be both
Lebesgue measure preserving and a homeomorphism when restricted to any orbit.

\subsection{Main results}
\label{sec:main-results}

In the present paper we prove an analog of DJK classification for free \( \mathbb{R}^{d} \)-flows.

\setcounter{section}{9}
\begin{theorem}
  \label{thm:main-theorem-intro}
  Free non smooth Borel \( \mathbb{R}^{d} \)-flows are Lebesgue Orbit Equivalent if and only
  if cardinalities of the sets of pie measures are the same.
\end{theorem}
\setcounter{section}{1}

The structure of the paper is as follows.  Section \ref{sec:esixt-cocomp-lacun} provides a
different proof of a theorem due to C.~Conley which guarantees
existence of cocompact cross sections.  Section
\ref{sec:orbit-equiv-locally} introduces various notions of orbit
equivalences which preserve measures between orbits and Section
\ref{sec:invar-meas-cross-phase} shows a correspondence between
invariant measures on a phase space and on a cocompact cross section.

Remaining sections deal with \( \mathbb{R}^{d} \)-flows and prove various parts of the main result.
In Section \ref{sec:rect-tilings-Rn} we introduce rectangular tilings, Section
\ref{sec:unif-rokhl-lemma} proves a version of Rokhlin lemma that works for all invariant measures
at the same time, and Section \ref{sec:regul-tilings-orbits} employs these results together with a
tiling method due to Rudolph \cite{rudolph_rectangular_1988} to construct a LOE between invariant
subsets of uniformly full measures.  Section \ref{sec:compr-flows-proof} then deals with the
complementary case of flows with no invariant measures and in Section \ref{sec:proof-main-theorem}
we finally prove the main theorem.

\section{Cocompact cross sections}
\label{sec:esixt-cocomp-lacun}

Let \( G \) be a locally compact Polish group acting in a Borel way on a standard Borel space
\( X \).  In this section the action is not assumed to be free.  A \emph{cross section} for
\( G \acts X \) is a Borel set \( \mathcal{C} \subseteq X \) that intersects every orbit of the
action, i.e., \( G \cdot \mathcal{C} = X \), and is such that for some neighborhood of the identity
\( U \subseteq G \) one has
\begin{equation}
  \label{eq:10}
  U \cdot x \, \cap \, U \cdot y = \es \quad \textrm{whenever \( x, y \in \mathcal{C} \) are distinct.}
\end{equation}

We will typically assume \( U \) to be symmetric and compact.  When for a set \( U \) the cross
section \( \mathcal{C} \) satisfies \eqref{eq:10}, we say that the cross section is
\emph{\( U \)-lacunary}.  Frequently the definition of a cross section is relaxed by omitting the
lacunarity condition and requiring the countability of intersections with orbits instead, but since
all the cross sections in our work will be lacunary, we choose to adopt the stronger
definition. 

Lacunarity says that distinct points within an orbit are never ``too close'' to each other.  The
complementary notion of not having ``large gaps'' within orbits is called cocompactness.  A cross
section \( \mathcal{C} \) is \emph{cocompact} if there exists a compact neighborhood of the identity
\( V \subseteq G \) such that \( V \cdot \mathcal{C} = X \), and \( \mathcal{C} \) is said to be
\( V \)-cocompact in that case.  Finally, we say that a cross section \( \mathcal{C} \) is a
\emph{maximal} \( U \)-lacunary cross section if it is \( U \)-lacunary and for any
\( z \in X \setminus \mathcal{C} \) the set \( \mathcal{C} \cup \{z\} \) is no longer a
\( U \)-lacunary cross section.  When \( U \) is a symmetric neighborhood of the identity,
a \( U \)-lacunary cross section is maximal if and
only if it is \( U^{2} \)-cocompact.

\addpicturenocap{r}{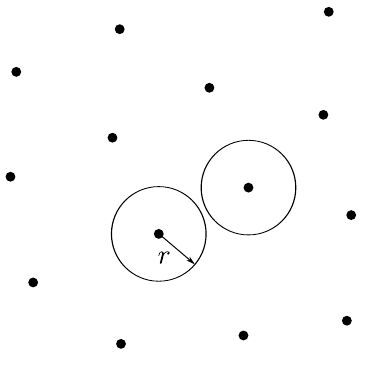}{39.558mm}{Lacunary and cocompact}
\stepcounter{figure}

To illustrate this concepts consider, for instance, the case \( G = \mathbb{R}^{2} \) and suppose
additionally that the action \( \mathbb{R}^{2} \acts X \) is free.  We may therefore identify each
orbit with an affine\footnote{By affine we mean a plane with no distinguished origin.  It
  nevertheless carries all the structures from \( \mathbb{R}^{2} \) which are translation invariant:
  Euclidean distance, Lebesgue measure, etc. } copy of the plane.  Take \( U \) to be a
ball of radius \( r \) around the origin in \( \mathbb{R}^{2} \).  A cross section
\( \mathcal{C} \subset X \) is \( U \)-lacunary if and only if balls of radius \( r \)
around distinct points in \( \mathcal{C} \) do not intersect.  It is maximal
\( U \)-lacunary if moreover any point in the orbit is at most \( 2r \) from a point in
\( \mathcal{C} \).

A theorem of Kechris \cite[Corollary 1.2]{kechris_countable_1992} establishes existence of
\( U \)-lacunary cross sections for any action \( G \acts X \) of a locally compact group and any
given compact neighborhood of the identity \( U \subseteq G \).
Clinton Conley proved that one can always find a cocompact cross
section.

\begin{theorem-nn}[Conley]
  Any Borel action of a locally compact group admits a cocompact cross section.
\end{theorem-nn}

Conley's argument uses \( G_{0} \)-dichotomy.
For the purpose of completeness we present in this section an
elementary construction of enlarging any cross section into a maximal
one.

Let \( \mathcal{C} \subseteq X \) be a \( W \)-lacunary cross section, where
\( W \subseteq G \) is a compact symmetric neighborhood of the identity.  We turn
\( \mathcal{C} \) into a Borel graph with the set of edges \( F_{W} \) by putting an edge
between \( x \) and \( y \) whenever there
exists \( g \in W \) such that \( gx = y \):
\[ F_{W} = \bigl\{ (x,y) \in \mathcal{C} \times \mathcal{C} : \exists g \in W\
gx = y \bigr\}. \]

Lacunarity of \( \mathcal{C} \) ensures that the
degree of any vertex in \( F_{W} \) is finite (in fact, the degree is uniformly bounded).  A subset
\( \mathcal{A} \subseteq \mathcal{C} \) is said to be \emph{\( F_{W} \)-independent} if
there are no edges between its points:
\[ x \ne y \implies (x,y) \not \in F_{W} \quad \textrm{for all \( x, y \in \mathcal{A} \)}
.\]

The proof of \cite[Lemma 1.17]{jackson_countable_2002} shows that \( \mathcal{C} \) can be written
as a disjoint union of countably many Borel \( F_{W} \)-independent subsets.  We reproduce
the argument for the reader's convenience.

\begin{lemma}
  \label{lem:F-independent-union}
  Let \( \mathcal{C} \subseteq X \) be a cross section and let \( W \subseteq G \) be a
  compact symmetric neighborhood of the identity.  There exist Borel
  \( F_{W} \)-independent subsets \( \mathcal{A}_{n} \subseteq \mathcal{C} \)
  such that \( \mathcal{C} = \bigsqcup_{n \in \mathbb{N}} \mathcal{A}_{n} \).
\end{lemma}

\begin{proof}
  Let \( (Z_{n})_{n=1}^{\infty} \) be a countable family of Borel subsets of \( \mathcal{C} \) which
  separates points and which is closed under finite intersections.  Since \( F_{W} \) is
  locally finite, for any \( x \in \mathcal{C} \) there exists \( n \) such that
  \( F_{W}[x] \cap Z_{n} = \{x\} \), where
  \[ F_{W}[x] = \{ y \in \mathcal{C} : (x,y) \in F_{W}\}. \]
  Let \( \phi : \mathcal{C} \to \mathbb{N} \) be given by
  \[ \phi(x) = \min \bigl\{\, n \in \mathbb{N} : F_{W}[x] \cap Z_{n} = \{x\}\, \bigr\}. \]
  The function \( \phi \) is Borel and \( \phi^{-1}(n) \) is \( F_{W} \)-independent for
  each \( n \).  We may therefore set \( \mathcal{A}_{n} = \phi^{-1}(n) \).
\end{proof}

We now need a refinement of the notion of maximality for a cross section.  Let \( \mathcal{C}
\subseteq X \) be a \( U \)-lacunary cross section.  Given a set \( Y
\subseteq X \), we say that \( \mathcal{C} \) is a \emph{maximal \( U \)-lacunary cross section
in \( Y \)} if for no \( y \in Y \setminus \mathcal{C} \) the set \( \mathcal{C} \cup \{y\} \) is \(
U \)-lacunary.  In other words, \( \mathcal{C} \) cannot be enlarged by an element from \(
Y \) while keeping \( U \)-lacunarity.

\begin{lemma}
  \label{lem:enlarging-with-discrete-base}
  Let \( U \) and \( V \) be compact symmetric neighborhoods of the identity in \( G \), and let
  \( \mathcal{C} \) be a \( U \)-lacunary cross section.
  Set \( W = V \cdot U^{2} \cdot V \), and let \( \mathcal{A} \subseteq \mathcal{C} \)
  be an \( F_{W} \)-independent Borel set.  There exists a maximal \( U \)-lacunary in
  \( V \cdot \mathcal{A} \)\, Borel cross section \( \mathcal{D} \) that moreover contains
  \( \mathcal{C} \).
\end{lemma}

\begin{proof}
  Let \( (f_{n}) \) be a countable family dense in \( V \), set
  \begin{displaymath}
    \begin{aligned}
      Y_{0} &= \{ x \in \mathcal{A} : f_{0}(x) \not \in U^{2} \cdot \mathcal{C} \},\\
      Y_{1} &= \bigl\{ x \in \mathcal{A} : f_{1}(x) \not \in U^{2} \cdot \bigl(\mathcal{C} \cup
      f_{0}(Y_{0})\bigr) \bigr\},\\
      Y_{2} &= \bigl\{ x \in \mathcal{A} : f_{2}(x) \not \in U^{2} \cdot \bigl(\mathcal{C} \cup
      f_{0}(Y_{0}) \cup f_{1}(Y_{1})\bigr)\bigr\},\\
      &\ ...\\
      Y_{k} &= \bigl\{ x \in \mathcal{A} : f_{k}(x) \not \in U^{2} \cdot \bigl( \mathcal{C}
      \cup \bigcup_{i < k}
      f_{i}(Y_{i}) \bigr)\bigr\}.\\
  \end{aligned}
\end{displaymath}
In words, \( Y_{0} \) contains all \( x \in \mathcal{A} \) such that \( f_{0}(x) \) can be added to
\( \mathcal{C} \) while preserving \( U \)-lacunarity; \( Y_{1} \) consists of those
\( x \in \mathcal{A} \) such that \( f_{1}(x) \) can be added to \( \mathcal{C} \cup f_{0}(Y_{0}) \)
keeping \( U \)-lacunarity, etc.

Note that \( \mathcal{C} \cup f_{0}(Y_{0}) \) is \( U \)-lacunary.  Indeed, suppose
\( y,z \in \mathcal{C} \cup f_{0}(Y_{0}) \) are such that \( g_{1}y = g_{2}z \) for some
\( g_{1}, g_{2} \in U \) and \( y \ne z \).  Since \( \mathcal{C} \) is
\( U \)-lacunary, at least one of \( y, z \) has to be in \( f_{0}(Y_{0}) \), say
\( z \in f_{0}(Y_{0}) \).  If \( y \in \mathcal{C} \), then \( z = g_{2}^{-1}g_{1} y \) implies
\( z \in U^{2} \cdot \mathcal{C} \), contradicting the construction of \( Y_{0} \).  Hence
\( y \in f_{0}(Y_{0}) \).  Let \( x_{y}, x_{z} \in Y_{0} \) be such that \( y = f_{0} x_{y} \),
\( z = f_{0} x_{z} \).  Equality \( z = g_{2}^{-1}g_{1} y \) implies
\( x_{z} = f_{0}^{-1} g_{2}^{-1}g_{1} f_{0} (x_{y}) \).  Since
\[ f_{0}^{-1} g_{2}^{-1}g_{1} f_{0} \in V \cdot U^{2} \cdot V =
W, \]
and \( x_{y}, x_{z} \in \mathcal{A} \), the set \( \mathcal{A} \) is not
\( F_{W} \)-independent.  Contradiction.

This shows that \( \mathcal{C} \cup f_{0}(Y_{0}) \) is \( U \)-lacunary.  Similarly one
checks that \( \mathcal{C} \cup f_{0}(Y_{0}) \cup f_{1}(Y_{1}) \) is \( U \)-lacunary, and
in fact \( \mathcal{D} = \mathcal{C} \cup \bigcup_{i} f_{i}(Y_{i}) \) is \( U \)-lacunary.  It
remains to see that \( \mathcal{D} \) is maximal \( U \)-lacunary in
\( V \cdot \mathcal{A} \).

To begin with, for any \( n \) and any \( x \in \mathcal{A} \) such that
\( f_{n}(x) \not \in \mathcal{D} \), \( \mathcal{D} \cup \{f_{n}(x)\} \) is not
\( U \)-lacunary, for otherwise
\( \mathcal{C} \cup \bigcup_{i < n} f_{i}(Y_{i}) \cup \{f_{n}(x)\} \) would be
\( U \)-lacunary, whence \( x \in Y_{n} \), implying that
\( f_{n}(x) \in f_{n}(Y_{n}) \subseteq \mathcal{D} \).  In other words, \( \mathcal{D} \) cannot be
enlarged to a \( U \)-lacunary cross section by adding an element of the form
\( f_{n}(x) \) for some \( x \in \mathcal{A} \).

Suppose there is some \( x \in \mathcal{A} \) and \( y \in V \cdot x \) such that
\( y \not \in \mathcal{D} \) and \( \mathcal{D} \cup \{y\} \) is \( U \)-lacunary.  Let
\( g \in V \) be such that \( gx = y \) and pick \( (n_{k})_{k=0}^{\infty} \) for which
\( f_{n_{k}} \to g \).  Since \( U \) is a neighborhood of the identity, we may assume
that \( f_{n_{k}} g^{-1} \in U \) for all \( k \); in particular
\( f_{n_{k}} x \not \in \mathcal{D} \) for all \( k \).  But we showed that \( \mathcal{D} \) cannot
be enlarged to a \( U \)-lacunary cross section by an element of the form
\( f_{n_{k}} x \), whence there are
\( z_{k} \in (U^{2} \cdot V \cdot x ) \cap \mathcal{D}\) such that
\( h_{k} f_{n_{k}} x = z_{k} \) for some \( h_{k} \in U^{2} \).  Since
\( (U^{2} \cdot V \cdot x) \cap \mathcal{D}\) is finite, by passing to a
subsequence we may assume that \( h_{k} f_{n_{k}} x = z \) for some fixed \( z \in \mathcal{D} \)
and all \( k \).  Recall that \( U \) is compact, and so by passing to a subsequence once
again, we may assume that \( h_{k} \to h \in U^{2} \).  Let \( p \in G \) be some element such
that \( p z = x \), for example \( p = (h_{0} f_{n_{0}})^{-1} \).  Thus
\( p h_{k} f_{n_{k}} x = x \) for all \( k \), i.e., elements \( p h_{k} f_{n_{k}} \) are in the
stabilizer of \( x \).  By Miller's Theorem \cite[Theorem 9.17]{kechris_classical_1995} stabilizers
are closed, and thus \( phg x = x \), hence \( h g x = p^{-1} x = z \).  But \( g x = y \),
\( h \in U^{2} \), and \( z \in \mathcal{D} \), whence \( \mathcal{D} \cup \{y\} \) is not
\( U \)-lacunary.  This proves the claim and the lemma.
\end{proof}

\begin{lemma}
  \label{lem:enlarging-over-K}
  Let \( U \) and \( V \) be a compact symmetric neighborhoods of the identity in \( G \), and let
  \( \mathcal{C} \) be a \( U \)-lacunary cross section.  There exists a \( U \)-lacunary cross
  section \( \mathcal{D} \) containing \( \mathcal{C} \) which is maximal \( U \)-lacunary in
  \( V \cdot \mathcal{C} \).
\end{lemma}

\begin{proof}
  Let \( W = V \cdot U^{2} \cdot V \).  By Lemma
  \ref{lem:F-independent-union} we may write \( \mathcal{C} = \bigsqcup_{n} \mathcal{A}_{n} \) with
  each \( \mathcal{A}_{n} \) being Borel and \( F_{W} \)-independent.  Using Lemma
  \ref{lem:enlarging-with-discrete-base} we construct inductively \( U \)-lacunary cross
  sections \( \mathcal{C} \subseteq \mathcal{C}_{0} \subseteq \mathcal{C}_{1} \subseteq \cdots \)
  such that \( \mathcal{C}_{i} \) is maximal \( U \)-lacunary in
  \( V \cdot \mathcal{A}_{i} \).  It is easy to see that
  \( \mathcal{D} = \bigcup_{n} \mathcal{C}_{n} \) is maximal \( U \)-lacunary in
  \( V \cdot \mathcal{C} \).
\end{proof}

\begin{theorem}
  \label{prop:existence-of-cocompact}
  Let \( U \) be a compact symmetric neighborhood of the identity in \( G \).  For every
  \( U \)-lacunary Borel cross section \( \mathcal{C} \subseteq X \) there exists a maximal Borel
  \( U \)-lacunary cross section \( \mathcal{D} \subseteq X \) such that moreover
  \( \mathcal{C} \subseteq \mathcal{D} \).  The cross section \( \mathcal{D} \) is therefore
  \( U^{2} \)-cocompact.
\end{theorem}

\begin{proof}
  Fix an increasing sequence \( (V_{n})_{n=1}^{\infty} \) of symmetric compact neighborhoods of the
  identity which cover \( G \), i. e., \( G = \bigcup_{n} V_{n} \).  Applying Lemma
  \ref{lem:enlarging-over-K}, construct \( U \)-lacunary cross sections
  \( \mathcal{C} = \mathcal{C}_{0} \subseteq \mathcal{C}_{1} \subseteq \mathcal{C}_{2} \subseteq
  \cdots \) such that \( \mathcal{C}_{i} \) is maximal \( U \)-lacunary in
  \( V_{i} \cdot \mathcal{C}_{i-1} \) (and in particular, \( \mathcal{C}_{i} \) is maximal
  \( U \)-lacunary in \( V_{i}\cdot \mathcal{C} \)).  We claim that
  \( \mathcal{D} = \bigcup_{n} \mathcal{C}_{n} \) is maximal \( U \)-lacunary, i.e., we claim that
  for no \( y \in X \setminus \mathcal{D} \) the cross section \( \mathcal{D} \cup \{y\} \) is
  \( U \)-lacunary.  Indeed, take \( y \in X \setminus \mathcal{D} \) and let
  \( x \in \mathcal{C} \) be such that \( y \in G x\).  Pick \( n \) so large that
  \( y \in V_{n} \cdot x \).  Since \( \mathcal{D} \) is maximal \( U \)-lacunary in
  \( V_{n} \cdot \mathcal{C} \), \( \mathcal{D} \cup \{y\} \) cannot be \( U \)-lacunary, and the
  claim follows.
\end{proof}

\section{Orbit equivalences of locally compact group actions}
\label{sec:orbit-equiv-locally}

From now on we consider only free actions of groups.

Let \( H \) be a locally compact group and let \( \lambda \) be a \emph{right} invariant Haar
measure on \( H \).  Suppose also that \( H \) acts freely on a standard Borel space \( X \).  Once
we select a point \( x \in X \), the orbit \( \orbit(x) \) can be identified with the group \( H \)
itself via
\[ H \ni h \mapsto hx \in \orbit(x). \]
This identification makes it possible to transfer the measure \( \lambda \) from \( H \) onto the
orbit \( \orbit(x) \) by setting for a set \( A \subseteq \orbit(x) \)
\[ \lambda_{x}(A) = \lambda \bigl(\{h \in H: hx \in A \}\bigr). \]
The right invariance of \( \lambda \) implies that \( \lambda_{x}(A) = \lambda_{y}(A) \) whenever
\( y \in \orbit(x) \):  if \( fx = y \), \( f \in H \), then
\[ \lambda_{y}(A) = \lambda \bigl( \{h \in H : h y \in A\} \bigr) = \lambda\bigl( \{h \in H : hfx
\in A\} \bigr) = [g := hf] = \lambda\bigl( \{ g \in H: gx \in A \} f^{-1} \bigr) =
\lambda_{x}(A), \]
and therefore the push forward of the Haar measure onto the orbit \( \orbit(x) \) is independent of
the base point.  When \( H \) is discrete, the measure on an orbit is just the counting measure, and
any orbit equivalence thus automatically preserves the counting measure.  This results in the fact
that any orbit equivalence induces a bijection between the sets of invariant measures on the phase
spaces.  In the non discrete world this is may no longer be the case.  To fix this, we introduce a
strengthening of OE which requires the isomorphism between phase spaces to preserve Haar measures on
orbits.  There are three increasingly restrictive rigorous formulations of this concept.

Let \( G_{1} \acts X_{1} \) and \( G_{2} \acts X_{2} \) be a pair of free Borel actions of locally
compact groups, and pick right invariant Haar measures \( \lambda_{1} \), \( \lambda_{2} \) on
\( G_{1} \) and \( G_{2} \) respectively.  We say that the actions \( G_{1} \acts X_{1} \) and
\( G_{2} \acts X_{2} \) are \emph{weakly Haar Orbit Equivalent} (abbreviated by wHOE) if there
exists an orbit equivalence \( \phi : X_{1} \to X_{2} \) such that for any \( x \in X_{1} \) the
push forward measure \( \phi_{*} \lambda_{1,x} \) is a multiple of \( \lambda_{2, \phi(x)} \).  In
other words, the actions are wHOE if there exists an orbit equivalence which sends a Haar measure on
each orbit onto a Haar measure.  Note that we allow the normalization to vary from orbit to orbit.
It is evident that the definition of wHOE does not depend on the choice of \( \lambda_{1} \) and
\( \lambda_{2} \).  Given a wHOE \( \phi : X_{1} \to X_{2} \) we may associate a
\emph{normalization} function
\[ \alpha_{\phi}^{\lambda_{1}, \lambda_{2}} = \alpha_{\phi} : X_{1} \to \mathbb{R}^{>0} \quad
\textrm{defined by the condition} \quad \phi_{*}\lambda_{1,x} = \alpha_{\phi}(x)
\lambda_{2,\phi(x)}. \]
Note that \( \alpha_{\phi}(x) = \alpha_{\phi}(y) \) whenever \( x \) and \( y \) are
\( E_{X_{1}}^{G_{1}} \)-equivalent.  The function \( \alpha_{\phi} \) does depend on the choice of
\( \lambda_{1} \) and \( \lambda_{2} \), but in a mild way: if
\( \alpha_{\phi}^{\lambda_{1}', \lambda_{2}'} \) is defined with respect to some other choice of
right Haar measures \( \lambda_{1}' \) and \( \lambda_{2}' \), then
\( \alpha_{\phi}^{\lambda_{1}', \lambda_{2}'} \) is a constant multiple of
\( \alpha_{\phi}^{\lambda_{1}, \lambda_{2}} \).  The property of \( \alpha_{\phi} \) being constant
is therefore independent of the choice of \( \lambda_{1} \) and \( \lambda_{2} \).  This allows us
to introduce the following definition.
\begin{definition}
  \label{def:haar-orbit-equivalence}
  Given a pair of free actions \( G_{1} \acts X_{1} \) and \( G_{2} \acts X_{2} \), we say that
  these actions are \emph{Haar Orbit Equivalent} (HOE) whenever there exists a weak Haar Orbit
  Equivalence \( \phi : X_{1} \to X_{2} \) with the constant normalization function: \(
  \alpha_{\phi}^{\lambda_{1}, \lambda_{2}} \equiv \mathrm{const} \) for some (equivalently, any)
  right Haar measures \( \lambda_{1} \), \( \lambda_{2} \) on \( G_{1} \) and \( G_{2} \) respectively.
\end{definition}
One could reformulate HOE by saying that given \( \lambda_{1} \) for a suitable choice of
\( \lambda_{2} \) the push forward \( \phi_{*}\lambda_{1,x} \) is equal to the measure
\( \lambda_{2, \phi(x)} \) for all \( x \in X_{1} \).

Sometimes one has a natural normalization choice of the Haar measure, the Euclidean space
\( \mathbb{R}^{d} \) being a notable example.  One may than wonder whether for a \emph{given choice}
of \( \lambda_{1} \) and \( \lambda_{2} \) on \( G_{1} \) and \( G_{2} \) there exists a HOE
\( \phi: X_{1} \to X_{2} \) such that the normalization function is constantly equal to one. This
brings us to our last and strongest definition.

\begin{definition}
  \label{def:lebesgue-orbit-equivalence}
  Let \( \lambda_{1} \) and \( \lambda_{2} \) be right Haar measures on \( G_{1} \) and \( G_{2} \),
  and let \( G_{1} \acts X_{1} \), \( G_{2} \acts X_{2} \) be a pair of free actions.  We say that
  these actions are \emph{Lebesgue Orbit Equivalent} (LOE) if there exists a HOE
  \( \phi : X_{1} \to X_{2} \) such that the corresponding normalization function is constantly equal to
  \( 1 \): \( \alpha_{\phi}^{\lambda_{1}, \lambda_{2}}(x) = 1 \) for all \( x \in X_{1} \).
\end{definition}

Equivalently, \( \phi_{*}\lambda_{1,x} = \lambda_{1,\phi(x)} \) holds for all \( x \in X_{1} \).
Whereas wHOE and HOE do not depend on the choice of \( \lambda_{1} \) and \( \lambda_{2} \), the
notion of LOE requires a choice of Haar measures.  We shall always equip Euclidean spaces
\( \mathbb{R}^{d} \) with the standard Lebesgue measure, and LOE between \( \mathbb{R}^{d} \)-flows
should always be understood with respect to that choice of Haar measures.

\section{Invariant measures on a cross sections and phase spaces}
\label{sec:invar-meas-cross-phase}

In this section \( G \) is assumed to be a unimodular\footnote{A locally compact
  group is \emph{unimodular} if its left Haar measure is also a right Haar measure, i.e., if the
  Haar measure is two-sided invariant.} locally compact Polish group acting freely on \( X \).  We
fix a Haar measure \( \lambda \) and a compatible proper\footnote{A compatible metric \( d \) on a
  locally compact group is \emph{proper} if all balls \( B_{r}(e) = \{ g \in G : d(g,e) \le r\} \)
  are compact. Existence of such metrics has been established in \cite{struble_metrics_1974}.} left
invariant metric \( d \) on \( G \).

\vspace*{2mm}
\addpicture{r}{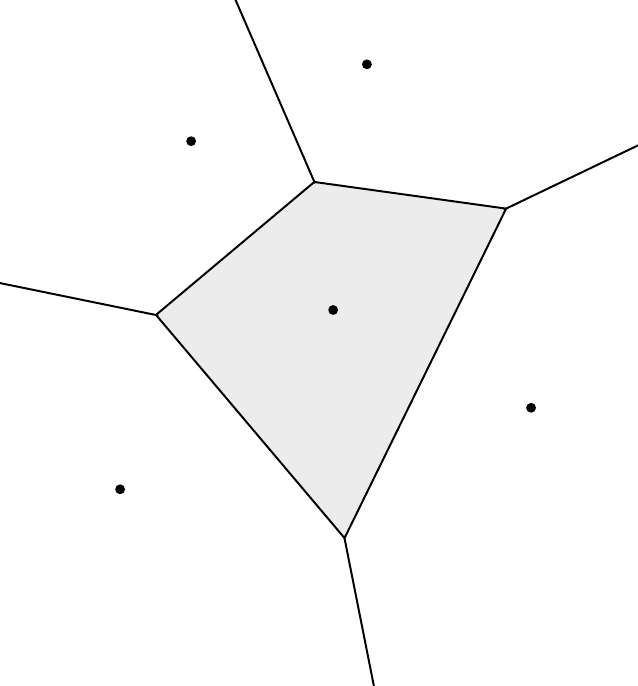}{64.759mm}{Voronoi
  tessellation\vspace{4mm}}{fig:voronoi-tessellation}
\subsection{Voronoi domains}
\label{sec:voronoi-domains}

\vspace*{-3mm}

Let \( E = E_{X}^{G} \) denote the orbit equivalence relation of the action:
\( (x, y) \in E \) whenever there exists some \( g \in G \) such that \( gx = y \).  Note that
freeness of the action implies uniqueness of such an element \( g \in G \).  The set \( E \) is
Borel and so is the function \( \rho : E \to G \) that assigns to a pair \( (x,y) \in E \) the unique
\( g \in G \) such that \( gx = y \).

Given a cross section for the action \( G \acts X \), we associate a decomposition of each orbit
into Voronoi domains, which are determined by the proximity to the points in the cross section as
measured by the metric \( d \).  Let us first recall the construction in the more familiar setting
of the Euclidean space followed by the formal definition of the general case.

Let \( \mathcal{C} \subseteq \mathbb{R}^{d} \) be a  cross section.  A
\emph{Voronoi domain} of a point \( c \in \mathcal{C} \) consists of all the points
\( x \in \mathbb{R}^{d} \) such that \( x \) is closer to \( c \) in the Euclidean distance than to
any other point in \( \mathcal{C} \).  Voronoi tessellation is the partition of \( \mathbb{R}^{d} \)
into Voronoi domains.  Figure~\ref{fig:voronoi-tessellation} shows a fragment of Voronoi tessellation on
\( \mathbb{R}^{2} \) determined by the five points.  The gray polygon is the Voronoi domain of the
central point.  Note that attribution of points on the boundary of domains is ambiguous.

In the general set up of a unimodular group action we start by fixing a Borel linear order
\( \preceq_{\mathcal{C}} \) on the cross section.  It will be used to distribute the boundary points
between the Voronoi domains.  The \emph{Voronoi tiling} determined by \( \mathcal{C} \) (also known
as Voronoi tessellation or Voronoi partition) is the set
\( \mathscr{V} \subseteq X \times \mathcal{C} \) given by
\[ \mathscr{V}= \Bigl\{ (x,c) : \forall a \in \mathcal{C} \cap (G \cdot c) \quad \Bigl(
d\bigl(\rho(x,c),e\bigr) < d\bigl(\rho(x,a),e\bigr) \Bigr) \textrm{ or } \Bigl(
d\bigl(\rho(x,c),e\bigr) = d\bigl( \rho(x,a),e \bigr) \textrm{ and } c \preceq_{\mathcal{C}} a
\Bigr) \Bigr\}. \]
In words, \( (x,c) \in \mathscr{V} \) if \( c \) is closer to \( x \) than any other point in
\( \mathcal{C} \), or, when there are several closest points in \( \mathcal{C} \), \( c \) is the
minimal one according to \( \preceq_{\mathcal{C}} \).  Note that by lacunarity, there can only be
finitely many closest points, so the minimal one always exists.  The \emph{Voronoi domain} of
\( c \in \mathcal{C} \) is the set
\[ \mathscr{V}_{c} = \{x \in X : (x,c) \in \mathscr{V} \}.\]
Note that if \( \mathcal{C} \) is cocompact, Voronoi domains \( \mathscr{V}_{c} \) are bounded (in the
sense that \( d(c, \,\cdot\,) : \mathscr{V}_{c} \to \mathbb{R} \) is bounded) uniformly in \( c \).

The set \( \mathscr{V} \) is an example of a tiling of the action \( G \acts X \).

\begin{definition}
  \label{def:general-tiling}
  A \emph{tiling} of \( G \acts X \) is a Borel set
  \( \mathscr{W} \subseteq X \times X \) satisfying the following properties.
  \begin{enumerate}[(i)]
  \item Projection \( \mathcal{C} \) of the set \( \mathscr{W} \) onto the second coordinate is
    Borel and forms a cross section of the action,
    \[ \mathcal{C} = \{c \in X: (x,c) \in X \textrm{ for some } x \in X\}. \]
  We refer to this cross section as the one \emph{associated} with \( \mathscr{W} \).  By definition \(
  \mathscr{W} \subseteq X \times \mathcal{C} \).
\item \( \mathscr{W} \subseteq E_{G} \), i.e., \( \mathscr{W} \) respects the orbit
  equivalence relation.
\item For any \( x \in X \) there exists a unique \( c \in \mathcal{C} \) with
  \( (x,c) \in \mathscr{W} \), i.e.,  \( \mathscr{W} \) is a graph of a Borel function
  \( w : X \to \mathcal{C} \).
  \item There exists a neighborhood of the identity \( U \subseteq G \) such that
    \( \mathcal{C} \) is \( U \)-lacunary and \( U \times \{c\} \subseteq \mathscr{W} \) for all
    \( c \in \mathcal{C} \).
  \end{enumerate}

  Given a tiling \( \mathscr{W} \), the domain of a point \( c \in \mathcal{C} \) is the set \(
  \mathscr{W}_{c} \) of all \( x \in X \) such that \( (x,c) \in \mathscr{W} \).  Geometrically any
  orbit \( \mathcal{O} \subseteq X \) is therefore tiled by sets \( \mathscr{W}_{c} \),  \( c \in
  \mathcal{O} \cap \mathcal{C}  \).

  We say that a tiling is \emph{bounded} if the cross section \( \mathcal{C} \) associated with \(
  \mathscr{W} \) is cocompact and domains \( \mathscr{W}_{c} \) are bounded uniformly in \( c \), in
  the sense that 
  there exists a compact set \( V \subseteq G \) such that \( \mathscr{W}_{c} \subseteq V \cdot c \)
  for all \( c \in \mathcal{C} \).
\end{definition}

Note that given any cocompact cross section \( \mathcal{C} \subseteq X \), the Voronoi tessellation
constructed above is an example of a bounded tiling of \( G \acts X \) with associated cross section
\( \mathcal{C} \).

Our goal in the next few subsections is to establish a correspondence between finite invariant
measures on the phase space \( X \) and on a cocompact cross section \( \mathcal{C} \).

\subsection{Lifting measures to the phase space}
\label{sec:lift-meas-phase}

Let \( \nu \) be a finite invariant measure on a cocompact cross section~\( \mathcal{C} \).  We are
going to lift \( \nu \) to a finite invariant measure \( \mu_{\nu} \) on \( X \).  For this we fix a
bounded tiling \( \mathscr{W} \) of \( G \acts X \) which associated cross section is
\( \mathcal{C} \); e.g., the Voronoi tessellation constructed above.  We
mention in advance that the measure \( \mu_{\nu} \) will turn out to be independent of the choice of
\( \mathscr{W} \).  Given a Borel set \( A \subseteq X \) and \( c \in \mathcal{C} \), let
\( \xi(A, c ) \) be defined by:
\[ \xi(A, c) = \lambda\bigl( \{ g \in G : gc \in \mathscr{W}_{c} \cap A \} \bigr). \]
For a fixed \( c \in \mathcal{C} \), \( \xi(\,\cdot\,, c) \) is a finite (albeit not invariant)
measure on \( X \).  Measures \( \xi(\,\cdot\,, c) \) are bounded uniformly in \( c \) since
\( \mathscr{W} \) is assumed to be bounded.  The measure \( \mu_{\nu} \) is then defined to be the
integration of \( \xi(\,\cdot\, , c) \) with respect to \( \nu \): for a Borel \( A \subseteq X \)
\[ \mu_{\nu}(A) = \int_{\mathcal{C}} \xi(A,c)\, d\nu(c). \]

We would like to show that \( \mu_{\nu} \) is \( G \)-invariant.  Let \( T : X \to X \) be a Borel
automorphism of the orbit equivalence relation \( E_{X}^{G} \), and suppose that \( T \)
preserves the Haar measure on each orbit: for any orbit \( \mathcal{O} \subseteq X \), any Borel
set \( P \subseteq \mathcal{O} \), and any (equivalently, some) \( x \in \mathcal{O} \) one has \(
\lambda_{x}(P) = \lambda_{x}(TP) \), or in a more detailed form,
\[ \lambda \bigl( \{g \in G : gx \in P \} \bigr)= \lambda \bigl( \{g \in G : gx \in TP\} \bigr). \]

\begin{proposition}
  \label{prop:mu-nu-is-T-invariant}
  In the notations above, \( \mu_{\nu}(TA) = \mu_{\nu}(A) \) for all Borel \( A \subseteq X \).
\end{proposition}
\begin{proof}
  Let \( H \) be a countable group, with an enumeration of its element \( (h_{n})_{n=1}^{\infty} \),
  such that the relation 
  \[ E_{\mathcal{C}} = E_{X}^{G} \cap (\mathcal{C} \times \mathcal{C})\]
  is given by an action \( H \acts \mathcal{C} \) (existence of such an action is guaranteed by a
  theorem of J.~Feldman and C.~C.~Moore \cite{feldman_ergodic_1977}).  Fix a Borel set
  \( A \subseteq X \) and let
  \[ A_{n}' = \{ x \in A : h_{n}w(x) = w(Tx) \}, \]
  where \( w : X \to \mathcal{C} \) is a function which graph is \( \mathscr{W} \).
  The sets \( A_{n}' \) may not be disjoint, so we also define
  \( A_{1} = A_{1}' \) and \( A_{n} = A_{n}' \setminus \bigcup_{k < n}  A_{k} \).  This gives us a
  partition \( A = \bigsqcup_{n} A_{n} \).  Since
  \[ \mu_{\nu}(A) = \sum_{n} \mu_{\nu}(A_{n}) \quad \textrm{and} \quad \mu_{\nu}(TA) = \sum_{n}
  \mu_{\nu}(TA_{n}), \]
  to prove the proposition it is enough to show that \(
  \mu_{\nu}(TA_{n}) = \mu_{\nu}(A_{n}) \) for each \( n \).

  Sets \( A_{n} \) are constructed in such a way that for any \( c \in \mathcal{C} \) one has
  \[ T( A_{n} \cap \mathscr{W}_{c}) = TA_{n} \cap \mathscr{W}_{h_{n}c}, \]
  and therefore the assumption that \( T \) preserves the Haar measure within each orbit implies
  \begin{equation}
    \label{eq:1}
    \xi(A_{n},c) = \xi(TA_{n}, h_{n}c) \quad \textrm{ for all \( c \in \mathcal{C} \).}
  \end{equation}
  We therefore conclude
  \[ \mu_{\nu}(TA_{n}) = \int_{\mathcal{C}} \xi(TA_{n},c) \, d\nu(c) = [\tilde{c} := h_{n}^{-1}c] =
  \int_{\mathcal{C}} \xi(TA_{n}, h_{n}\tilde{c})\, d\nu(h_{n}\tilde{c}) = \int_{\mathcal{C}}
  \xi(A_{n}, \tilde{c})\, d\nu(\tilde{c}) = \mu_{\nu}(A_{n}), \]
  where the penultimate equality follows from \eqref{eq:1} and from the \( H \)-invariance of
  \( \nu \).
\end{proof}

\begin{corollary}
  \label{cor:mu-nu-is-G-invariant}
  The measure \( \mu_{\nu} \) is \( G \)-invariant.
\end{corollary}
\begin{proof}
  For \( g \in G \) set \( T_{g}(A) = gA \) and apply Proposition \ref{prop:mu-nu-is-T-invariant}.
\end{proof}

Note that the construction of \( \mu_{\nu} \) and Proposition \ref{prop:mu-nu-is-T-invariant} are
valid when \( G \) is not necessarily unimodular and \( \lambda \) is merely a right invariant Haar
measure, but Corollary \ref{cor:mu-nu-is-G-invariant} uses left invariance of \( \lambda \).

For future reference we note that if a neighborhood of the identity \( U \subseteq G \) is such that
\( U \times \{c\} \subseteq \mathscr{W} \) for all \( c \in \mathcal{C} \), then the map
\( U \times \mathcal{C} \ni (g, c) \mapsto gc \in U \cdot \mathcal{C} \) is a bijection, and via
this identification
\begin{equation}
  \label{eq:2}
  \mu_{\nu} |_{U \cdot \mathcal{C}} = \lambda|_{U} \times \nu.
\end{equation}

\subsection{Pulling measures to a cross section}
\label{sec:pull-meas-cross}
One can pull measures in the opposite direction as well.  Pick \( U \subseteq G \) to be such a
small neighborhood of the identity that \( \mathcal{C} \) is \( U \)-lacunary.  If \( \mu \) is an
invariant measure on the phase space \( X \), we set for \( \mathcal{A} \subseteq \mathcal{C} \) the measure
\( \nu_{\mu}(\mathcal{A}) \) to be defined by
\[ \nu_{\mu}(\mathcal{A}) = \frac{\mu(U \cdot \mathcal{A})}{\lambda(U)}. \]

The definition turns out to be independent of the choice of \( U \).  The measure
\( \nu_{\mu} \) is a finite invariant measure on \( \mathcal{C} \), and moreover,
\begin{equation}
  \label{eq:3}
  \mu|_{U \cdot \mathcal{C}} = \lambda|_{U} \times \nu_{\mu}.
\end{equation}
For a slick proof of these assertions see \cite[Proposition 4.3]{kyed_l^2-betti_2013}.  Whereas
cocompactness of \( \mathcal{C} \) was used in Subsection \ref{sec:lift-meas-phase} to ensure
that \( \mu_{\nu} \) is finite, results of this subsection are valid even when \( \mathcal{C} \) is
not necessarily cocompact.

\subsection{Correspondence between ergodic measures}
\label{sec:corr-betw-ergod}

The maps \( \mu \mapsto \nu_{\mu} \) and \( \nu \mapsto \mu_{\nu} \) are inverses of each other.
Indeed, \eqref{eq:2} and \eqref{eq:3} imply that \( \nu_{\mu_{\nu}} \) is such that via the natural
identifications
\[ \lambda|_{U} \times \nu_{\mu_{\nu}} = \mu_{\nu}|_{U \cdot \mathcal{C}} =
\lambda|_{U} \times \nu, \]
and therefore \( \nu_{\mu_{\nu}} = \nu  \) for any finite invariant measure \( \nu \) on \(
\mathcal{C} \).  On the other hand, for any invariant measure \( \mu \) on \( X \)
\[ \mu_{\nu_{\mu}}|_{U \cdot \mathcal{C}} = \lambda|_{U} \times \nu_{\mu} =
\mu|_{U \cdot \mathcal{C}}, \]
implying that \( \mu_{\nu_{\mu}}|_{U \cdot \mathcal{C}} = \mu|_{U \cdot
  \mathcal{C}} \), but both measures are \( G \)-invariant, whence \( \mu_{\nu_{\mu}} = \mu \).  In
particular, the construction of \( \mu_{\nu} \) does not depend on the choice of a bounded tiling \(
\mathscr{W} \).  

While the map \( \mu \mapsto \nu_{\mu} \) is a linear bijection between spaces of invariant
measures, it does not preserve the normalization: in general
\( \mu(X) \ne \nu_{\mu}(\mathcal{C}) \).  Let \( \mathcal{E}(X) \) denote the family of all pie
measures on \( X \) (recall that pie stands for probability invariant ergodic); invariance and
ergodicity is understood to be with respect to the orbit equivalence relations \( E_{X}^{G} \).
Similarly let \( \mathcal{E}(\mathcal{C}) \) denote the family of pie measures on \( \mathcal{C} \)
with respect to the relation \( E_{\mathcal{C}} \).
\begin{proposition}
  \label{prop:bijection-pie-measures}
  The map \( \mathcal{E}(X) \ni \mu \mapsto \nu_{\mu}/\nu_{\mu}(\mathcal{C}) \in
  \mathcal{E}(\mathcal{C}) \) is a bijection between \( \mathcal{E}(X) \) and \(
  \mathcal{E}(\mathcal{C}) \).  
\end{proposition}
\begin{proof}
  We start by checking that for any finite invariant measure \( \mu \) on \( X \) the measure
  \( \nu_{\mu} \) is ergodic if and only if \( \mu \) is.  Indeed, suppose \( \nu_{\mu} \) is
  ergodic and let \( Z \subseteq X \) be a \( G \)-invariant set.  Since \( Z \cap \mathcal{C} \) is
  \( E_{\mathcal{C}} \)-invariant, either \( \nu_{\mu}(Z \cap \mathcal{C}) = 0 \) or
  \( \nu_{\mu}(\mathcal{C} \setminus Z) = 0 \).  Suppose for definiteness that the former is
  realized.  By \eqref{eq:3}, \( \mu\bigl(U \cdot (Z \cap \mathcal{C})\bigr) = 0 \), and therefore
  \( \mu(Z) = 0 \).  Thus \( \mu \) must be ergodic.

  If \( \mu \) is ergodic and \( Z \subseteq \mathcal{C} \) is \( E_{\mathcal{C}} \)-invariant, then
  either \( \mu(G \cdot Z) = 0 \) or \( \mu\bigl(G \cdot (\mathcal{C} \setminus Z)\bigr) = 0 \).
  Equation \eqref{eq:3} implies that either \( \nu_{\mu}(Z) = 0 \) or
  \( \nu_{\mu}(\mathcal{C} \setminus Z) = 0\) must take place.

  This proves that for any \( \mu \in \mathcal{E}(X) \) the measure
  \( \nu_{\mu}/ \nu_{\mu}(\mathcal{C}) \) is indeed an element of \( \mathcal{E}(\mathcal{C}) \).
  This map is injective for if \( \mu_{1} \ne \mu_{2} \) are pie measure on \( X \), there exists an
  invariant \( Z \subseteq X \) such that \( \mu_{1}(Z) = 1 \) and \( \mu_{2}(Z) = 0 \).  This implies
  that \( \nu_{\mu_{1}}(Z \cap \mathcal{C})  > 0 \) and \( \nu_{\mu_{2}}(Z \cap \mathcal{C}) = 0 \),
  therefore \( \nu_{\mu_{1}}/ \nu_{\mu_{1}}(\mathcal{C}) \) and \( \nu_{\mu_{2}}/
  \nu_{\mu_{2}}(\mathcal{C}) \) are distinct.

  Finally the map is surjective.  Indeed, for a given \( \nu \in \mathcal{E}(\mathcal{C}) \) there
  exists a finite invariant measure \( \mu \) on \( X \) such that \( \nu_{\mu} = \nu \).  By the
  above, the measure \( \mu \) must be ergodic, hence \( \mu/ \mu(X) \in \mathcal{E}(X) \), but also
  \[ \nu_{a\mu} = a\nu_{\mu} \quad \textrm{ for all } a \in \mathbb{R}^{>0}, \]
  and therefore 
  \[ \frac{\nu_{\mu/\mu(X)}}{\nu_{\mu/\mu(X)}(\mathcal{C})} =
  \frac{\nu_{\mu}}{\nu_{\mu}(\mathcal{C})} = \nu. \qedhere \]
\end{proof}

We conclude this section with an observation that wHOE induces a bijection between pie measures on
the phase spaces.

\begin{theorem}
  \label{thm:wHOE-bijection-pie}
  Let \( G_{1} \acts X_{1} \) and \( G_{2} \acts X_{2} \) be free Borel actions of unimodular
  locally compact Polish groups \( G_{1} \) and \( G_{2} \), and let \( \phi : X_{1} \to X_{2} \) be
  a wHOE between the actions.  The push-forward map
  \( \phi_{*} : \mathcal{E}(X_{1}) \to \mathcal{E}(X_{2}) \) induces a bijection between pie
  measures on \( X_{1} \) and \( X_{2} \).
\end{theorem}

\begin{proof}
  The only thing that is not immediate from the definitions is that \( \phi_{*}\mu \) is \( G_{2}
  \)-invariant whenever \( \mu \) is a \( G_{1} \)-invariant measure.  Let \( Z \subseteq X_{2} \)
  and let \( h \in G_{2} \).  We need to show that \( \phi_{*}\mu(Z) = \phi_{*}(hZ) \), or in other
  words,
  \[ \mu\bigl( \phi^{-1}(Z) \bigr) = \mu\bigl( \phi^{-1}(hZ) \bigr). \]
  Let \( T : X_{1} \to X_{1} \) be defined by \( Tx = \phi^{-1} \circ h \circ \phi(x) \).  The map
  \( T \) is a Borel bijection preserving \( E_{X_{1}}^{G_{1}} \) and moreover, we claim that \( T \)
  preserves the Haar measure on each orbit of \( G_{1} \), satisfying therefore the assumptions of
  Proposition \ref{prop:mu-nu-is-T-invariant}.  Once this claim is proved, we get
  \[ \mu\bigl( \phi^{-1}(Z) \bigr) = \mu\bigl( T \phi^{-1}(Z) \bigr) = \mu\bigl( \phi^{-1}(hZ)
  \bigr).\]

  To see that \( T \) preserves the Haar measure, let us pick an orbit
  \( \mathcal{O} \subseteq X_{1} \), a point \( x_{1} \in \mathcal{O} \), and let \( \lambda_{1} \),
  \( \lambda_{2} \) be Haar measures on \( G_{1} \) and \( G_{2} \) respectively.  By the definition
  of wHOE, \( \phi_{*}\lambda_{1,x} = \alpha_{\phi}(x) \lambda_{2,\phi(x)} \) for some
  \( \alpha_{\phi}(x) \in \mathbb{R}^{>0} \) which is moreover constant on \( \mathcal{O} \).  For a
  subset \( P \subseteq \mathcal{O} \) using that \( \lambda_{2,\phi(x)} \) is \( G_{2} \)-invariant
  we have
  \[ \lambda_{1,x}( T P ) = \lambda_{1,x}\bigl( \phi^{-1}\circ h \circ \phi(P) \bigr) = \alpha_{\phi}(x)
  \lambda_{2,\phi(x)}\bigl(h \circ \phi(P)\bigr) = \alpha_{\phi}(x) \lambda_{2,\phi(x)}\bigl( \phi(P) \bigr) =
  \lambda_{1,x} (P). \]
  Thus \( \lambda_{1,x}(TP) = \lambda_{1,x}(P) \) for all \( P \subseteq \mathcal{O} \) and the
  theorem follows.
\end{proof}

\section{Rectangular tilings of multidimensional flows}
\label{sec:rect-tilings-Rn}

From this section onward we restrict ourselves to the case of a free Borel action
\( \mathbb{R}^{d} \acts X \) of the Euclidean space on a standard Borel space \( X \).  Recall that
we use an additive notation for the action: for \( x \in X \) and \( r \in \mathbb{R}^{d} \) the
action of \( r \) upon \( x \) is denoted by \( x + r \).

Our main tool in understanding multidimensional flows is the concept of a rectangular tiling.
Simply put, it is a Borel partition of orbits into rectangles\footnote{Perhaps it would be
  more accurate to speak of \emph{cuboids} rather than rectangles, but since our figures illustrate
  the case \( d = 2 \), and since the argument is the same in all dimensions, we choose to use the
  two-dimensional terminology throughout the paper.}\kern-2.6mm.

\begin{definition}
  \label{def:tiling}
  A \emph{rectangular tiling} of an action \( \mathbb{R}^{d} \acts X \) is a tiling \( \mathscr{R}
  \) (in the sense of Definition \ref{def:general-tiling}) such that for the associated cross section \(
  \mathcal{C} \) and any \( c \in \mathcal{C} \) the domain \( \mathscr{R}_{c} \) is a (half-open)
  rectangle: \( \mathscr{R}_{c} = c + R_{c} \), where \( R_{c} \) is of the form 
  \[ R_{c} = \prod_{i=1}^{d}[a_{i}, b_{i}).\]
  Equivalently, a rectangular tiling is a cross section \( \mathcal{C} \subseteq X \) together with
  bounded away from zero Borel functions (we call them \emph{dimension functions})
  \( \zeta_{i}^{l}, \zeta_{i}^{r} : \mathcal{C} \to \mathbb{R}^{>0} \), \( i \le d \), such that
  rectangles \( R_{c} = \prod_{i=1}^{d}\bigl[-\zeta_{i}^{l}(c), \zeta_{i}^{r}(c)\bigr) \) tile all
  the orbits: for any orbit \( \mathcal{O} \subseteq X \)
  \[ \mathcal{O} = \bigsqcup_{c \in \mathcal{C} \cap \mathcal{O}} (c + R_{c}).\]
\end{definition}

\begin{figure}[ht]
  \centering
  \includegraphics[width=65.889mm]{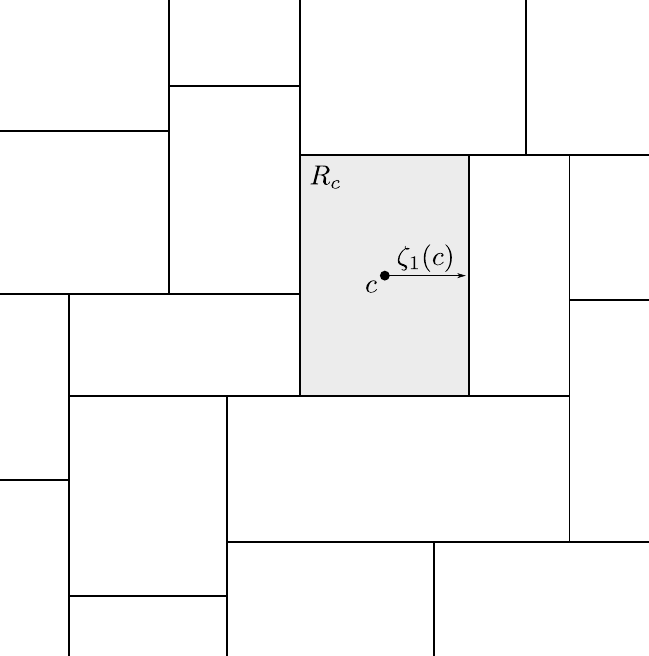}    
  \caption{Rectangular tiling}
  \label{fig:tiling}
\end{figure}

Figure \ref{fig:tiling} shows how portion of an orbit may look like.  The gray rectangle corresponds
to a single tile of the form \( c + R_{c} \).  All our rectangles are half-open to ensure that
\( c + R_{c} \) and \( c' + R_{c'} \) are disjoint whenever \( c \ne c' \).  Note also that given a
tiling \( \mathscr{R} \), we may select in a Borel way centers of tiles \( c + \vec{w}_{c} \), where
\( \vec{w}_{c}(i) = \bigl(\zeta^{r}_{i}(c) - \zeta^{l}_{i}(c)\bigr)/2 \).  Unless stated otherwise,
we shall therefore assume that our tilings are symmetric and points \( c \in \mathcal{C} \) are
centers of the tiles \( R_{c} = \prod_{i=1}^{d} \bigl[-\zeta_{i}(c), \zeta_{i}(c) \bigr) \).

Existence of rectangular tilings for actions of \( \mathbb{Z}^{d} \) has been established by S.~Gao
and S.~Jackson \cite{gao_countable_2015}.  They proved that any free action
\( \mathbb{Z}^{d} \acts X \) admits a rectangular tiling, and moreover, for any
\( L \in \mathbb{R} \) one may always find a tiling \( \mathscr{R} \) such that the dimension
functions are bounded below by \( L \).  A similar result is true for \( \mathbb{R}^{d} \)-flows.

\begin{theorem}[Gao--Jackson for \( \mathbb{Z}^{d} \) actions]
  \label{thm:existence-of-rectangular-tilings}
  For any \( L > 0 \) there exists a rectangular tiling of \( \mathbb{R}^{d} \acts X \) with all
  edges of rectangles at least \( 2L \): \( \zeta_{i}(c) \ge L \) for all \( i \le d \) and all
  \( c \) in the associated cross section \( \mathcal{C} \).
\end{theorem}

What matters for the argument in \cite[Section 3]{gao_countable_2015} is the large scale
geometry of \( \mathbb{Z}^{d} \), which is the same as that of \( \mathbb{R}^{d} \), thus only
superficial modifications for their proof are required, which we therefore omit.

A standard large marker -- small maker argument allows us to improve the above statement by imposing
further restrictions on the dimension functions.

\begin{theorem}
  \label{thm:existence-of-rectangular-tilings-restr}
  Let \( \mathbb{R}^{d} \acts X \) be a free multidimensional flow.  For any \( L' > 0 \) and any
  \( \epsilon > 0 \) there exists a rectangular tiling \( \mathscr{Q} \) of the flow such that
  \[ |\zeta_{i}(c) - L'| < \epsilon\]
  for all \( c \) in the associated cross section \( \mathcal{C} \).
\end{theorem}

\begin{proof}
  Pick \( L \) so large that any real \( r \ge L \) can be partitioned into reals
  \( \epsilon \)-close to \( L \): for any \( r \ge L \) there exist \( s_{1}, \ldots, s_{n} > 0 \)
  such that \( r = \sum_{i=1}^{n}s_{i} \) and \( |s_{i} - L'| < \epsilon \) for all \( i \le n \).
  Use Theorem \ref{thm:existence-of-rectangular-tilings} and construct a tiling \( \mathscr{R} \)
  with all edges of tiles being at least \( L \).  By the choice of \( L \), each tile
  \( \mathscr{R}_{c} \) can be partitioned into rectangles \( Q_{c,1}, \ldots, Q_{c,n} \) having
  all edges \( \epsilon \)-close to \( L' \).  These rectangles \( Q_{c,i} \) constitute tiles of
  the desired tiling \( \mathscr{Q} \).
\end{proof}

\section{Uniform Rokhlin's Lemma}
\label{sec:unif-rokhl-lemma}

The following theorem is the usual Rokhlin's Lemma \cite[Theorem 1]{lind_locally_1975} when the
measure \( \mu \) is fixed.  The adjective ``uniform'' refers to the fact that
\( \mu(\mathcal{C} + R) > 1 - \epsilon \) holds for all invariant probability measures.

\begin{theorem}[Uniform Rokhlin's Lemma for \( \mathbb{R}^{d} \) actions]
  \label{thm:uniform-rokhlin-lemma}
  Let \( \mathbb{R}^{d} \acts X \) be a free Borel action.  For any rectangle
  \( R = \prod_{i=1}^{d}[-a_{i}, a_{i}) \), \( a_{i} > 0 \), and any \( \epsilon > 0 \) there exists
  a Borel \( R \)-lacunary cocompact cross section \( \mathcal{C} \subseteq X \) such that
  \( \mu(\mathcal{C} + R) > 1 - \epsilon \) for any invariant probability measure \( \mu \) on
  \( X \).
\end{theorem}

\begin{proof}
  We begin with an application of Theorem \ref{thm:existence-of-rectangular-tilings-restr} and
  select a tiling \( \mathscr{Q} \) with associated cross section \( \mathcal{C}_{0} \subseteq X \)
  and domains \( \mathscr{Q}_{c} = c + Q_{c} \), such that each rectangle \( Q_{c} \) is ``large''
  compared to \( R \).

  \begin{figure}[ht]
    \centering
    \includegraphics[width=44.803mm]{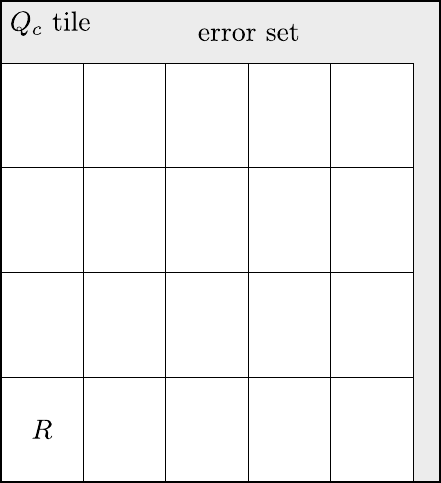}
    \caption{Large marker -- small marker}
    \label{fig:large-marker-small-marker}
  \end{figure}
  Now try to tile each \( c + Q_{c} \) with copies of \( R \) starting from the ``bottom left''
  corner as showed on Figure \ref{fig:large-marker-small-marker}.  Since lengths of edges of
  \( Q_{c} \) are not necessarily multiples of the corresponding edges of \( R \), there will be
  some remainder, aka error set, marked gray in Figure \ref{fig:large-marker-small-marker}.  The
  condition on the size of \( Q_{c} \) is that the proportion of the measure of the error set to the
  measure of \( \mathscr{Q}_{c} \) is at most \( \epsilon \): in the notation of Section
  \ref{sec:lift-meas-phase}
  \[ \xi(\textrm{error set},c)/\xi(\mathscr{Q}_{c},c) < \epsilon. \]
  Let \( \mathcal{C} \) denote the cross section which consists of centers of all the rectangles
  \( R \) inscribed into tiles \( \mathscr{Q}_{c} \).  We claim that \( \mathcal{C} \) satisfies the
  desired properties.

  The only thing that requires checking is that \( \mu(\mathcal{C} + R) > 1 - \epsilon \) for all
  probability \( E_{X} \)-invariant measures \( \mu \) on \( X \).  By Section
  \ref{sec:invar-meas-cross-phase} one can find a finite invariant measure \( \nu \) on
  \( \mathcal{C}_{0} \) such that
  \[ \mu = \int_{\mathcal{C}} \xi(\, \cdot\, , c)\, d\nu(c), \quad \textrm{where \( \xi(\,
    \cdot\, , c) \) is the ``Lebesgue'' measure on \( \mathscr{Q}_{c} \)}. \]
  By the construction of \( \mathcal{C} \), one has
  \( \xi(\mathcal{C} + R, c) > (1-\epsilon)\xi(\mathscr{Q}_{c}, c) = (1-\epsilon)\xi(X, c) \) and
  therefore
  \[ \mu(\mathcal{C} + R) = \int_{\mathcal{C}_{0}} \xi(\mathcal{C} + R,c)\, d\nu(c) >
  \int_{\mathcal{C}_{0}} (1-\epsilon) \xi(X,c)\, d\nu(c) = (1-\epsilon) \mu(X) =
  1-\epsilon. \qedhere \]
\end{proof}

\addpicturenocap{r}{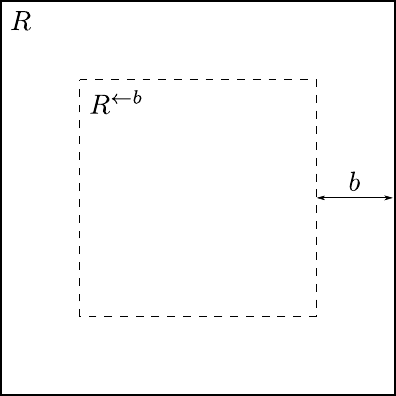}{40.2mm}{}
\stepcounter{figure}
For the proof of the main technical result, namely Theorem
\ref{thm:rokhlin-towers} below, we shall need an easy lemma.  Suppose we have a square
\( R = \prod_{i=1}^{d}[-l,l) \) and an \( R \)-lacunary cross section \( \mathcal{C} \).  If we
shrink \( R \) to a square \( R^{\la b} = \prod_{i=1}^{d}[-l+b, l-b) \), where \( b \) is small
compared to \( l \), then \( \mu (\mathcal{C} + R^{\la b}) \) has to be close to
\( \mu(\mathcal{C} + R) \) for all invariant measures \( \mu \) on \( X \).

In general, for any rectangle \( R = \prod_{i=1}^{d}[l_{i},r_{i}) \) and
\( b \in \mathbb{R}^{\ge 0} \) we let \( R^{\la b} \) denote the rectangle
\( \prod_{i=1}^{d}[l_{i}+b, r_{i}-b) \).  This notation will only be used when \( l_{i}+ b \) is
still less than \( r_{i} - b \).  In that case \( R^{\la b} \) is obtained by shrinking each edge of
\( R \) by \( b \) (see Figure to the right).  Note that if \( \widetilde{R} \) is any rectangle
contained in \( \prod_{i=1}^{d}[-L, L] \), then
\begin{equation}
  \label{eq:7}
 \hspace*{2.5cm} R^{\la b}  + \widetilde{R} \subseteq R^{\la b - L}.
\end{equation}

\vspace*{0.1mm}

\begin{lemma}
  \label{lem:small-boundary}
  Let \( \mathbb{R}^{d} \acts X \) be a free Borel flow.  For any \( \epsilon > 0 \), any real
  \( b \in \mathbb{R}^{\ge 0} \), there exists \( L \ge b \) such that for any \( l \ge L \), the square
  \( R = [-l,l)^{d} \), any \( R \)-lacunary cocompact cross section \( \mathcal{C} \subseteq X \),
  and any probability invariant measure \( \mu \) on \( X \) one has
  \[ \mu(\mathcal{C} + R^{\la b}) > \mu(\mathcal{C} + R) - \epsilon. \]
\end{lemma}

\begin{proof}
  Pick \( L \) so large that \( \lambda \bigl( R^{\la b} \bigr)/\lambda(R) > 1 - \epsilon \) for any
  \( R = [-l,l)^{d} \), \( l \ge L \), where
  \( \lambda \) is the Lebesgue measure on \( \mathbb{R}^{d} \).  We claim that such \( L \) works.
  Let \( \mu \) be a probability invariant measure on \( X \).  Since \( \mathcal{C} + R \) may be a
  proper subsets of \( X \), it does not form a tiling of the action.  It is, nevertheless, easy to
  enlarge it into a tiling as follows.  Let \( \mathscr{V} \) be the Voronoi tiling determined by \(
  \mathcal{C} \).  Define \( \mathscr{W} \) to be
  \[ \mathscr{W} = \Bigl\{ (x,c) : (x \in c + R) \textrm{ or } \Bigl( (x,c) \in \mathscr{V} \textrm{
    and } \bigl((x,c) \not \in c' + R \textrm{ for any } c' \in \mathcal{C}\bigr) \Bigr) \Bigr\}. \]
  It is easy to see that \( \mathscr{W} \) is a bounded tiling of \( \mathbb{R}^{d} \acts X \) and
  we may therefore apply results of Section \ref{sec:invar-meas-cross-phase} to decompose \( \mu \)
  as an integral over some measure \( \nu \) on \( \mathcal{C} \) of ``Lebesgue'' measures
  \( \xi(\, \cdot\, , c) \):
  \[ \mu = \int_{\mathcal{C}} \xi(\, \cdot\, , c)\, d\nu(c). \]
  Since \( \lambda(R^{\la b}) > (1-\epsilon)\lambda(R) \) we get
  \begin{displaymath}
    \begin{aligned}
      \mu(\mathcal{C} + R^{\la b}) &= \int_{\mathcal{C}} \xi(\mathcal{C} + R^{\la b}, c)\, d\nu(c) \\
      &> (1-\epsilon) \int_{\mathcal{C}} \xi(\mathcal{C} + R,c)\, d\nu(c)\\
      &= (1-\epsilon) \mu(\mathcal{C} + R) \\
      &\ge \mu(\mathcal{C} + R) - \epsilon \mu(X). 
    \end{aligned}
  \end{displaymath}
  Since the measure \( \mu \) is assumed to be a probability measure, \( \mu(X) = 1 \) and \(
  \mu(\mathcal{C} + R^{\la b}) > \mu(\mathcal{C} + R) - \epsilon \).
\end{proof}

In Ergodic Theory Rokhlin's Lemma is frequently applied countably many times to build a cover of the
phase space with a sequence of refining Rokhlin towers.  Exact details vary from application to
application, and the following theorem provides the set up that will be needed in our case.

\begin{theorem}
  \label{thm:rokhlin-towers}
  For any increasing sequence \( (b_{n})_{n=1}^{\infty} \), any real \( \kappa > 0 \), there exist
  an invariant Borel \( Z \subseteq X \), an increasing sequence of reals
  \( (l_{n})_{n=1}^{\infty} \), and a sequence of Borel sets \( \mathcal{C}_{n} \subseteq Z \)
  such that for \( R_{n} = [-l_{n}, l_{n})^{d} \) one has
  \begin{enumerate}[(i)]
  \item\label{item:Cn-is-cross-section-of-Z}
    \( (c + R_{n+1}) \cap \mathcal{C}_{n} \ne \es \) for each
    \( c \in \mathcal{C}_{n+1} \).
  \item\label{item:Z-union-Cn} \( Z = \bigcup_{n}(\mathcal{C}_{n} + R_{n}) \).
  \item\label{item:ln-multiple-kappa} Each \( l_{n} \) is an integer multiple of \( \kappa \).
  \item\label{item:ln-ge-bn} \( l_{n} \ge b_{n} \).
  \item\label{item:Cn-is-Rn-lacunary} \( \mathcal{C}_{n} \) is \( R_{n} \)-lacunary.
  \item\label{item:far-from-boundary}
    \( \mathcal{C}_{n} + R_{n} \subseteq \mathcal{C}_{n+1} + R_{n+1}^{\la b_{n+1}} \).
  \item\label{item:Z-uniformly-full-measure} \( \mu(Z) = 1 \) for any probability invariant measure
    \( \mu \) on \( X \).
  \end{enumerate}
\end{theorem}

Item \eqref{item:far-from-boundary} is the most important one here.  It says that each rectangle in
\( \mathcal{C}_{n} + R_{n} \) is inside a rectangle from \( \mathcal{C}_{n+1} + R_{n+1} \), and
moreover, it is far from its boundary.  Figure \ref{fig:rokhlin-towers} gives an illustration of
this item.  While it would be convenient to have such a covering on all orbits, this is not always
possible, and item \eqref{item:Z-uniformly-full-measure} offers the next best thing instead.
Reasons for taking \(l_{n} \) to be a multiple of \( \kappa \) will be apparent in the proof of
Theorem \ref{thm:Rudolphs-tiling}, but this restriction is not essential at any rate.  Note also
that items \eqref{item:Cn-is-cross-section-of-Z} and \eqref{item:Z-union-Cn} imply that each \(
\mathcal{C}_{n} \) is a cross section of \( \mathbb{R}^{d} \acts Z \).
\begin{figure}[ht]
  \centering
  \includegraphics[width=147.534mm]{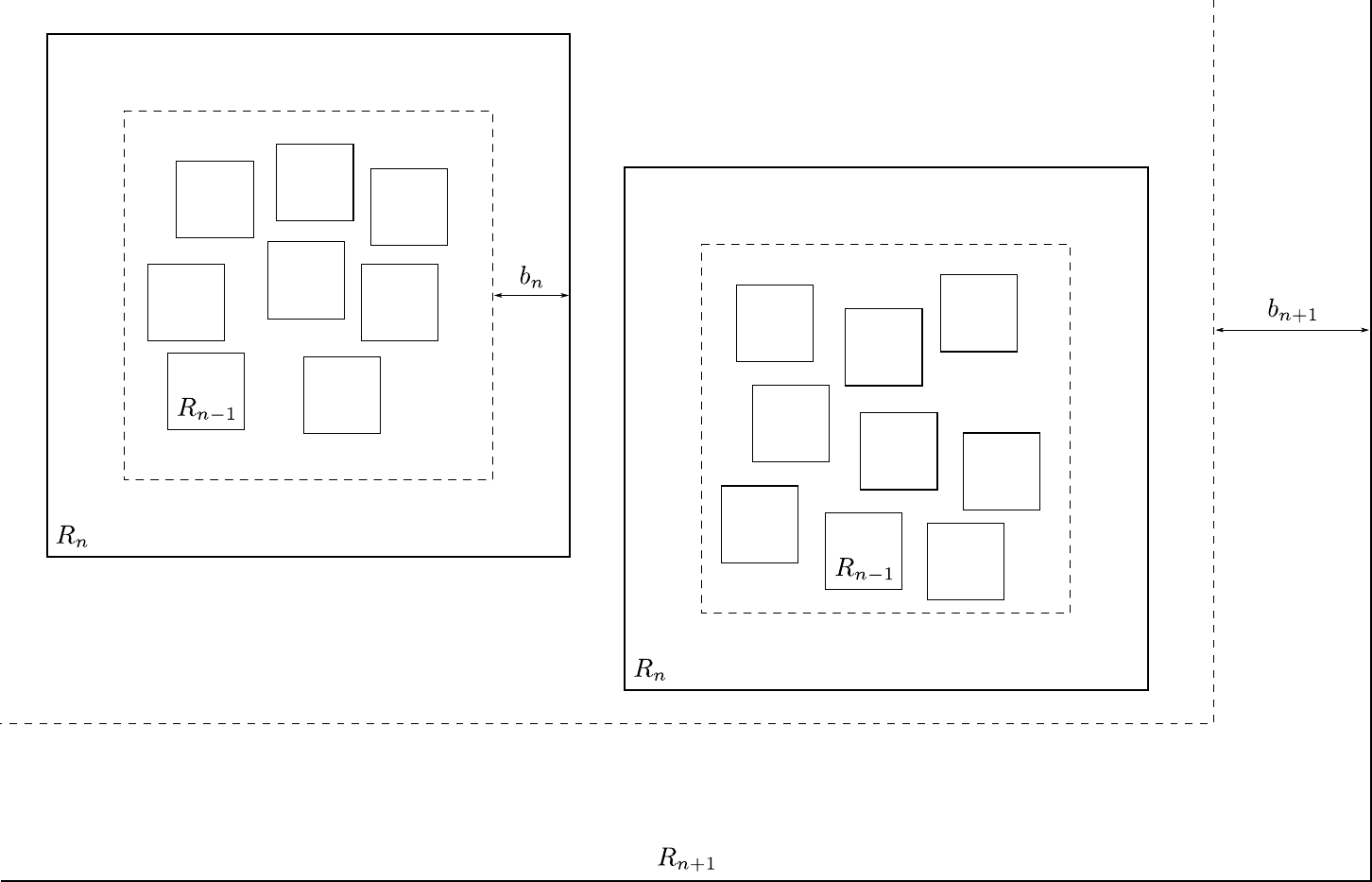}
  \caption{Rokhlin towers}
  \label{fig:rokhlin-towers}
\end{figure}

\begin{proof}
  Without loss of generality we assume that \( \lim_{n \to \infty} b_{n} = \infty \).  Pick a
  decreasing sequence \( (\epsilon_{n})_{n=1}^{\infty} \), \( \epsilon_{n} > 0 \), such that
  \( \sum_{n} \epsilon_{n} < 1 \).  Using Theorem \ref{thm:uniform-rokhlin-lemma} and Lemma
  \ref{lem:small-boundary} at each step, we construct inductively reals \( \tilde{b}_{n} \),
  \( l_{n} \), and \( R_{n} \)-lacunary cross sections \( \mathcal{C}_{n}' \subseteq X \), where \(
  R_{n} = [-l_{n}, l_{n})^{d} \), such that
  \begin{enumerate}[(a)]
  \item \( l_{n} \) is a multiple of \( \kappa \).
  \item \( l_{n} \ge \tilde{b}_{n} \ge b_{n} + 2l_{n-1} \) for all \( n \);
  \item\label{item:large-measure} \( \mu(\mathcal{C}_{n}' + R_{n}^{\la \tilde{b}_{n}}) > 1 - \epsilon_{n} \).
  \end{enumerate}
  For the base of the construction \( l_{0} \) is as assumed to be \( 0 \).

  Set for all \( k \ge 1 \)
  \[ \mathcal{C}_{k} = \mathcal{C}_{k}' \cap \Bigl( \mkern-8mu \bigcap_{n \ge k+1} \mkern-8mu \bigl(
  \mathcal{C}_{n}' + R_{n}^{\la \tilde{b}_{n} - l_{k}} \bigr) \Bigr), \]
  and let \( Z = \bigcup_{k} (\mathcal{C}_{k} + R_{k}) \).  We claim that sets
  \( \mathcal{C}_{k} \) and \( Z \) satisfy all the requirements of the theorem except, possibly,
  item \eqref{item:Cn-is-cross-section-of-Z}.  (It will be easy to enlarge sets \( \mathcal{C}_{n}
  \) to fulfill this item).

  Items (\ref{item:Z-union-Cn}-\ref{item:Cn-is-Rn-lacunary}) are evident from the construction.
  We check \eqref{item:far-from-boundary} next.  Pick \( x \in \mathcal{C}_{k} \) and note
  that by the definition of \( \mathcal{C}_{k} \), there exists \( y \in \mathcal{C}_{k+1}' \) such that
  \begin{equation}
    \label{eq:8}
    x \in y + R_{k+1}^{\la\tilde{b}_{k+1} - l_{k}}.
  \end{equation}

  Since \( \tilde{b}_{k+1} \ge b_{k+1} + 2l_{k} \), using \eqref{eq:7} we conclude that
  \[ x + R_{k} \subseteq y + R_{k+1}^{\la \tilde{b}_{k+1} - l_{k}}\! + R_{k} \subseteq y +
  R_{k+1}^{\la \tilde{b}_{k+1} - 2l_{k}} \subseteq y + R_{k+1}^{\la b_{k+1}}. \]
  To verify \eqref{item:far-from-boundary} it is therefore enough to check that this
  \( y \in \mathcal{C}_{k+1}' \) is actually an element of \( \mathcal{C}_{k+1} \).  Pick
  \( n > k+1 \); we show \( y \in \mathcal{C}_{n}' + R_{n}^{\la \tilde{b}_{n} - l_{k+1}} \).  Since
  \( x \in \mathcal{C}_{k} \), there exists \( z \in \mathcal{C}_{n}' \) such that
  \begin{equation}
    \label{eq:9}
    x \in z + R_{n}^{\la \tilde{b}_{n} - l_{k}}.    
  \end{equation}
  Using \( \tilde{b}_{k+1} -2l_{k} \ge 0 \), and equations \eqref{eq:8} and \eqref{eq:9}, we have
  the following chain of inclusions
  \[ y \in z + R_{n}^{\la \tilde{b}_{n} - l_{k}}\! - R_{k+1}^{\la\tilde{b}_{k+1} - l_{k}} \subseteq z
  + R_{n}^{\la \tilde{b}_{n} - l_{k} - (l_{k+1} - \tilde{b}_{k+1} + l_{k})} \subseteq z + R_{n}^{\la
  \tilde{b}_{n} - l_{k+1} + \tilde{b}_{k+1} - 2l_{k}} \subseteq z + R_{n}^{\la
    \tilde{b}_{n} - l_{k+1}}.  \]
  Thus \( y \in \mathcal{C}_{n}' + R_{n}^{\la \tilde{b}_{n} - l_{k+1}}\) for all \( n > k+1 \), and
  therefore \( y \in \mathcal{C}_{k} \), as required.  This checks \eqref{item:far-from-boundary}.

  Note that \eqref{item:far-from-boundary} implies that
  \( Z = \bigcup_{k}(\mathcal{C}_{k} + R_{k}) \) is an invariant subset of \( X \).  Indeed, for any
  \( c \in \mathcal{C}_{k} \) and any \( x \in c + \mathbb{R}^{d} \) we may find \( n \ge k \) so
  large that \( x \in c + R_{n} \) (recall that we assume \( \lim b_{n} = \infty \) and therefore
  also \( \lim l_{n} = \infty \)).  By item \eqref{item:far-from-boundary} we have
  \[ c \in \mathcal{C}_{k+1} + R_{k+1}^{\la b_{k+1}} \subseteq \mathcal{C}_{k+2} + R_{k+2}^{\la
    b_{k+1} + b_{k+2}} \subseteq \cdots \subseteq \mathcal{C}_{k+m} + R_{k+m}^{\la
    \sum_{i=1}^{m}b_{k+i}}.\]
  For \( m \) so large that \( \sum_{i=1}^{m} b_{k+i} \ge l_{n} \) we have
  \[ x \in c + R_{n} \subseteq \mathcal{C}_{k+m} + R_{k+m}^{\la \sum_{i=1}^{m}b_{k+i} - l_{n}}
  \subseteq \mathcal{C}_{k+m} + R_{k+m} \subseteq Z. \]

  To see item \eqref{item:Z-uniformly-full-measure}, note first that for any \( k \)
  \[ \bigcap_{n \ge k} \bigl( \mathcal{C}_{n}' + R_{n}^{\la \tilde{b}_{n}} \bigr) \subseteq \bigl(
  \mathcal{C}_{k}' + R_{k}\bigr) \cap \Bigl( \mkern-8mu \bigcap_{n \ge k+1}
  \mkern-8mu(\mathcal{C}_{n}' + R_{n}^{\la \tilde{b}_{n}}) \Bigr) \subseteq \biggl( \mathcal{C}_{k}'
  \cap \Bigl( \mkern-8mu \bigcap_{n \ge k+1} \mkern-8mu (\mathcal{C}_{n}' + R_{n}^{\la \tilde{b}_{n}
    - l_{k}}) \Bigr)\biggr) + R_{k} = \mathcal{C}_{k} + R_{k}.\]
  And in particular, for all \( k \ge 1 \) and all measures \( \mu \)
  \[ \mu\Bigl( \bigcap_{n \ge k} \bigl( \mathcal{C}_{n}' + R_{n}^{\la \tilde{b}_{n}} \bigr) \Bigr) \le
  \mu(\mathcal{C}_{k} + R_{k}). \]
  Whence by item \eqref{item:large-measure} in the construction of \( \mathcal{C}_{k} \),
  \[ \mu(Z) = \mu\bigl(\, \bigcup_{k} (\mathcal{C}_{k} + R_{k}) \bigr) \ge \sup_{k}
  \mu(\mathcal{C}_{k} + R_{k}) \ge \sup_{k} \mu\Bigl( \bigcap_{n \ge k} \bigl( \mathcal{C}_{n}' +
  R_{n}^{\la \tilde{b}_{n}} \bigr) \Bigr) \ge \lim_{k \to \infty} \Bigl( 1 - \sum_{n=k}^{\infty}
  \epsilon_{k} \Bigr) = 1.\]

  This almost finishes the proof of the theorem.  The only problem is item
  \eqref{item:Cn-is-cross-section-of-Z}.  It is possible to have points \( c \in \mathcal{C}_{n} \)
  such that the rectangle \( c + R_{n} \) has no points from \( \mathcal{C}_{k} \) for \( k < n \).
  The easy fix is to add all such \( c \in \mathcal{C}_{n} \) to \( \mathcal{C}_{k} \) for
  \( k < n \), i.e., we set
  \[ \bar{\mathcal{C}}_{k} = \mathcal{C}_{k} \cup \bigcup_{n \ge k+1}\{z \in \mathcal{C}_{n} :
  \mathcal{C}_{m} \cap (z + R_{n}) = \es \textrm{ for all } m < n\}. \]
  This enlargement does not violate any of the items
  (\ref{item:Z-union-Cn}-\ref{item:Z-uniformly-full-measure}), and sets \( \bar{\mathcal{C}}_{k} \) are
  as desired.
\end{proof}
\section{Rudolph's regular tilings}
\label{sec:regul-tilings-orbits}

In this section we employ Theorem \ref{thm:rokhlin-towers} and construct a regular tiling of a
subspace \( Z \subseteq X \) of the uniformly full measure.  Regularity refers to the fact that
orbits will be tiled by rectangles of finitely many shapes.  To describe possible tiles, we first
pick an irrational number \( \alpha > 0 \), for instance \( \alpha = \sqrt{2} \) will be good
enough.  For a vector \( \vec{a} \in \{1, \alpha\}^{d} \) let
\[ \widetilde{R}_{\vec{a}} = \prod_{i=1}^{d} \bigl[-a(i)/2,a(i)/2\bigr), \]
in other words each edge of \( \widetilde{R}_{\vec{a}} \) has length \( 1 \) or \( \alpha \).  We
let \( \one \) denote the vector \( (1,\dotsc,1) \in \{1,\alpha\}^{d} \) and \( \widetilde{R}_{\one} \) is
therefore a square with side one.

\begin{figure}[ht]
  \centering
  \includegraphics[width=115.776mm]{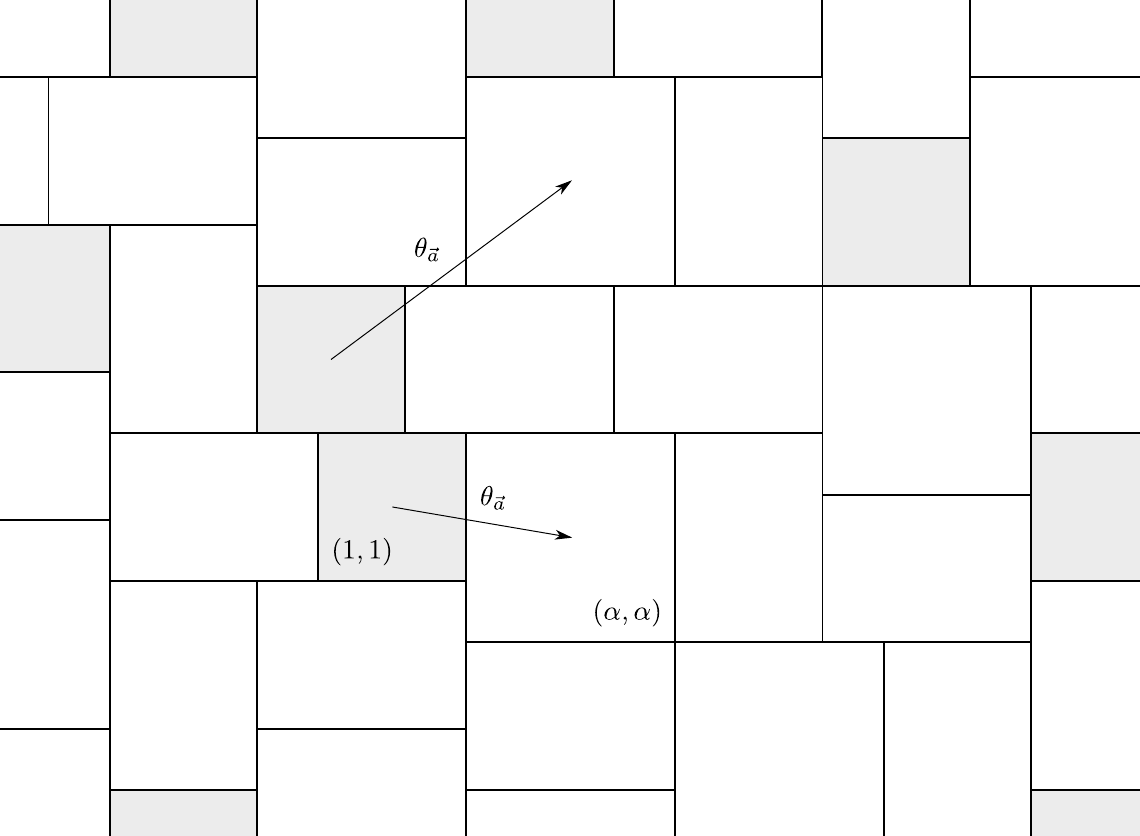}
  \caption{Regular tiling of an orbit.  There are four types of tiles. The map
    \( \theta_{\vec{a}} \), \( \vec{a} = (\alpha, \alpha) \), is a match between \( (1,1) \)-tiles
    (which are in gray) and \( (\alpha,\alpha) \)-tiles.}
  \label{fig:regular-tiling}
\end{figure}

Using Theorem \ref{thm:rokhlin-towers}, one can extract the following result from \cite[Section
3]{rudolph_rectangular_1988}.

\begin{theorem}[Essentially Rudolph]
  \label{thm:Rudolphs-tiling}
  Given a free Borel flow \( \mathbb{R}^{d} \acts X \) and an irrational \( \alpha > 0 \), there
  exists an invariant subset \( Z \subseteq X \) of uniformly full measure such that for the
  restriction of the action \( \mathbb{R}^{d} \acts Z \) the following holds.  There exists a
  rectangular tiling \( \mathscr{R} \) of \( \mathbb{R}^{d} \acts Z \) with associated cross section
  \( \mathcal{C} \), a Borel partition
  \( \mathcal{C} = \bigsqcup_{\vec{a} \in \{1,\alpha\}^{d}} \mathcal{C}_{\vec{a}} \mkern1mu\), and
  Borel bijections \( \theta_{\vec{a}} : \mathcal{C}_{\one} \to \mathcal{C}_{\vec{a}} \mkern1mu\),
  \( \vec{a} \in \{1, \alpha\}^{d} \setminus \{\one\} \), such that for
  \( \mathscr{R}_{c} = c + R_{c} \) one has
  \begin{enumerate}[(i)]
  \item\label{item:tiles-are-regular} \( R_{c} = \widetilde{R}_{\vec{a}}\)\, for all
    \( c \in \mathcal{C}_{\vec{a}} \).
  \item\label{item:bijection-between-tile-types} \( c\, E_{\mathcal{C}}\, \theta_{\vec{a}} (c) \)\, for
    all \( c \in \mathcal{C} \) and all \( \vec{a} \in \{1, \alpha\}^{d} \setminus \{\one\} \).
  \end{enumerate}
\end{theorem}

This theorem asserts that we can find a tiling of \( Z \) which uses only \( 2^{d} \) different
tiles.  It is generally easy to construct tilings with approximate properties of tiles, but getting
exact restrictions on length of edges is typically a more difficult task.  Naturally we would
want to have such a tiling on all of \( X \), but as of today, it is open whether this can always be
achieved.

During the construction of the tiling, we shall ensure that each type of tile occurs on each orbit
infinitely often.  In fact, we shall have a Borel witness for that, namely Borel
matchings \( \theta_{\vec{a}} : \mathcal{C}_{\one} \to \mathcal{C}_{\vec{a}} \) between tiles of
type \( \one \) and of type \( \vec{a} \) within each orbit.  Figure \ref{fig:regular-tiling} shows
how such a tiling may look like.

While the construction in \cite[Section 3]{rudolph_rectangular_1988} is presented relative to single
measure \( \mu \) on \( X \), it only uses existence of exhausting towers given in Theorem
\ref{thm:rokhlin-towers}.  Existence of Borel correspondences \( \theta_{\vec{a}} \) is almost
immediate from the construction.  For convenience of the reader, we present a sketch of the
argument.  While we believe that the reader will have no difficulties in supplying the necessary
details, if needed the rigorous proof can be found in \cite[Section 3]{rudolph_rectangular_1988}.

\begin{proof}[Sketch of Proof]
  The starting point is to notice that irrationality of \( \alpha \) implies that the set of points
  \[ \{\, m_{1} + m_{2} \alpha \,|\, m_{1}, m_{2} \in \mathbb{N} \,\} \]
  is \emph{asymptotically dense} in the real line \( \mathbb{R} \) in the sense that for any
  \( \epsilon > 0 \) there exists \( N(\epsilon) \) such that for any \( x \ge N(\epsilon) \) there
  are \( m_{1}, m_{2} \in \mathbb{N} \) for which \( |x - (m_{1} + m_{2}\alpha)| < \epsilon \).
  Geometrically this means that any sufficiently long interval after being perturbed by a small
  \( \epsilon \) can be partitioned in to sub-intervals of lengths \( 1 \) or \( \alpha \).

  Imagine now the following situation depicted on Figure \ref{fig:simple-tile-extension}.  Suppose
  we have a square \( R \) with side \( K(1+\alpha) \) for some integer \( K \).  Suppose also that
  \( R \) sits inside a much larger square \( R' \) with side length \( K' (1 + \alpha ) \) for some
  (much larger) integer \( K' \); suppose furthermore that the distance from \( R \) to the boundary
  of \( R' \) is at least \( N(\epsilon) \) in every coordinate direction.
  
    For notational convenience let us place the origin at the bottom left corner of \( R' \), so
  \( R' = [0, K' + K'\alpha)^{d} \), and let
  \[ R = \prod_{i=1}^{d} [a_{i}, b_{i}),\quad 0 < a_{i} < b_{i} < K + K\alpha. \]
  Since we assume that \( R \) is far from the boundary of \( R' \), \( a_{i} \ge N(\epsilon) \).
  We therefore may move \( R \) along the \( x \)-axis by a small \( \delta_{1} \),
  \( |\delta_{1}| \le \epsilon \), in such a way that \( a_{1} + \delta_{1} = m_{1} + m_{2} \alpha \) for
  some \( m_{1},m_{2} \in \mathbb{N} \).  The interval \( I_{1} = [0, a_{1} + \delta_{1}) \) may thus be
  tiled by intervals of length \( 1 \) and \( \alpha \).  The interval
  \[ I_{2} = [a_{1} + \delta_{1}, a_{1} +\delta_{1} + K + K\alpha) = [a_{1} + \delta, b_{1}+ \delta_{1}) \]
  can also be tiled in such a way since we assume that sides of \( R \) have length
  \( K(1 + \alpha) \).  Finally, the interval
  \[ I_{3} = [a_{1} + \delta_{1} + K + K \alpha, K' + K' \alpha) \]
  may be partitioned into segments of lengths \( 1 \) and \( \alpha \) since \( R' \) is supposed to
  have length \( K' + K' \alpha \); in particular 
  \[ K' + K' \alpha - a_{1} - \delta_{1} - K - K\alpha = \tilde{m}_{1} + \tilde{m}_{2}\alpha \]
  for some integers \( \tilde{m}_{1}, \tilde{m}_{2} \in \mathbb{N} \).

  \begin{figure}[ht]
    \centering
    \includegraphics[width=124.214mm]{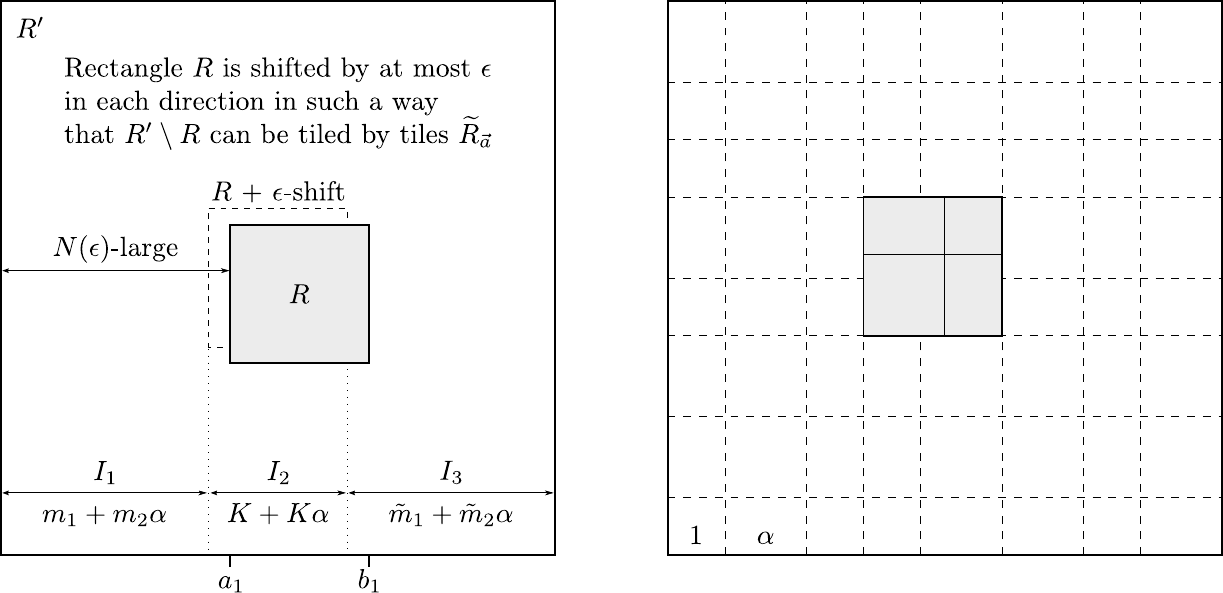}
    \caption{Moving the rectangle \( R \) and extending the tiling to \( R' \).}
    \label{fig:simple-tile-extension}
  \end{figure}

  The same can be done along other coordinate directions, and one may find a vector \( \vec{v} \) of
  \( \ell^{\infty} \)-norm at most \( \epsilon \) such that once \( R \) is shifted to
  \( R + \vec{v} \), the rectangle \( R' \) can be tiled by regular tiles
  \( \widetilde{R}_{\vec{a}} \), \( \vec{a} \in \{1, \alpha\}^{d} \), in a way that is consistent
  with the rectangle \( R + \vec{v} \).

  To summarize, given a square \( R \) of length \( K(1+\alpha) \) which is tiled by rectangles
  \( \widetilde{R}_{\vec{a}} \), \( \vec{a} \in \{1, \alpha\}^{d} \), and which is
  \( N(\epsilon) \)-far from the boundary of \( R' \), we may shift \( R \) by an
  \( \epsilon \)-small vector and extend the tiling of (shifted) \( R \) to a regular tiling of
  \( R' \).  The right square in Figure \ref{fig:simple-tile-extension} shows how such an extension
  may look like.

  Now to the construction of the regular tiling.  Let \( (\epsilon_{k})_{k=1}^{\infty} \) be a
  sufficiently fast decreasing sequence, e.g., \( \epsilon_{k} = 2^{-k} \); let
  \( N(\epsilon_{k}) \in \mathbb{R}^{>0} \) be so large that for any \( x \ge N(\epsilon_{k}) \)
  there exists \( m_{1}, m_{2} \in \mathbb{N} \) such that
  \[ |x - m_{1} - m_{2}\alpha| < \epsilon_{k}. \]
  Pick a sequence \( (b_{k})_{k=1}^{\infty} \) such that
  \begin{enumerate}[(a)]
  \item \( b_{k} \) is a multiple of \( (1+\alpha) \);
  \item \( b_{k} \ge N(\epsilon_{k}) + 2(1+\alpha) \).
  \end{enumerate}
  An application of Theorem \ref{thm:rokhlin-towers} allows us to find the following objects:
  \begin{itemize}
  \item an invariant subset \( Z \subseteq X \) of uniformly full measure;
  \item an increasing sequence \( (l_{k})_{k=1}^{\infty} \) of reals, \( l_{k } \ge b_{k} \), where
    each \( l_{k} \) is a multiple of \( (1+\alpha) \);
  \item Borel \( R_{k} \)-lacunary cross sections \( \mathcal{C}_{k} \subseteq Z \) such that
    \( \mathcal{C}_{k} + R_{k} \subseteq \mathcal{C}_{k+1} + R^{\la b_{k}} \), where
    \( R_{k} = [-l_{k},l_{k})^{d} \).
  \end{itemize}
  Our plan is to inductively tile regions \( \mathcal{C}_{k} + R_{k}^{\la b_{k}} \).  At step \( k+1 \)
  of the construction we may shift tiles in \( \mathcal{C}_{k} + R_{k}^{\la b_{k}} \) by at most \(
  \epsilon_{k} \), thus ensuring that the sum of all shifts is finite, and each tile converges to a
  limiting position as \( k \to \infty \).

  \addpicturenocap{r}{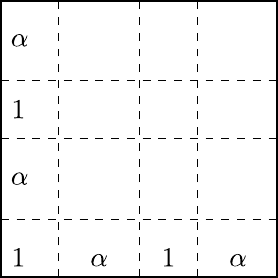}{28.316mm}{}
  \stepcounter{figure}
  Observe that since \( b_{k} \) and \( l_{k} \) are
  multiples of \( (1+\alpha) \), sides of \( R^{\la b_{k}} \) are also multiples of
  \( (1+\alpha) \).  And therefore \( R^{\la b_{k}} \) can be partitioned into tiles
  \( \widetilde{R}_{\vec{a}} \) in a canonical fashion by partitioning each side of
  \( R^{\la b_{k}} \) into consecutive intervals of length \( 1 \) and \( \alpha \) and taking the
  product of these partitions (see figure to the left).  We shall refer to this partition as to the
  \emph{canonical tiling} of \( R^{\la b_{k}} \).

  In particular, for the base of our construction each square \( c + R_{1}^{\la b_{1}} \) in
  \( \mathcal{C}_{1} + R_{1}^{\la b_{1}} \) can be tiled in this canonical way.

  For the induction step, we have tiled \( \mathcal{C}_{k} + R_{k}^{\la b_{k}} \), and proceed to
  tile \( \mathcal{C}_{k+1} + R_{k+1}^{\la b_{k+1}} \).  Pick some \( c \in \mathcal{C}_{k+1} \),
  and consider the corresponding square \( c + R_{k+1}^{\la b_{k+1}} \).  In general, it contains
  several points \( c_{1}, \dotsc, c_{m} \in \mathcal{C}_{k} \).  Each square
  \( c_{i} + R_{k}^{\la b_{k}} \) has been tiled, and we seek an extension of this tiling to a
  tiling of \( c + R_{k}^{\la b_{k+1}} \).  It is helpful to consult at this point Figure
  \ref{fig:extending-tiling-general-case}, on which \( m = 3 \) and three \( R_{k}^{\la b_{k}} \)
  rectangles are marked in gray.  By assumption, squares \( c_{i} + R_{k} \) are all inside
  \( c + R_{k+1}^{\la b_{k+1}} \) and have pairwise empty intersections.
  \( R_{k+1}^{\la b_{k+1}} \) admits the canonical tiling, indicated by dashed lines in Figure
  \ref{fig:extending-tiling-general-case}.  We shall use it to ``feel the gaps'' between
  squares~\( R_{k} \).  Now comes the crucial idea in the construction.  One realizes that it is
  always possible to move squares~\( R_{k} \) around \( c_{i} \) by at most \( 1 + \alpha \) in such
  a way that corners of \( R_{k} \) will coincide with nodes of the canonical tiling of
  \( R_{k+1}^{\la b_{k+1}} \).  We emphasize that tiles of \( \mathcal{C}_{k} + R_{k}^{\la b_{k}}\)
  constructed up to this stage are not being moved, we are rather claiming that one may select
  ``windows'' of size \( R_{k} \) around each point \( c_{i} \) with corners on the nodes of the
  canonical tiling of \( R_{k+1}^{\la b_{k+1}} \) and these ``windows'' are no further than
  \( 1+\alpha \) in each coordinate direction from \( c_{i} + R_{k} \).  It is easy to see that such
  squares may be chosen to be disjoint for distinct \( c_{i} \).  These windows \( R_{k} \) are
  depicted on Figure \ref{fig:extending-tiling-general-case} around each \( R_{k}^{\la b_{k}} \).

    \begin{figure}[ht]
      \centering
      \includegraphics[width=109.215mm]{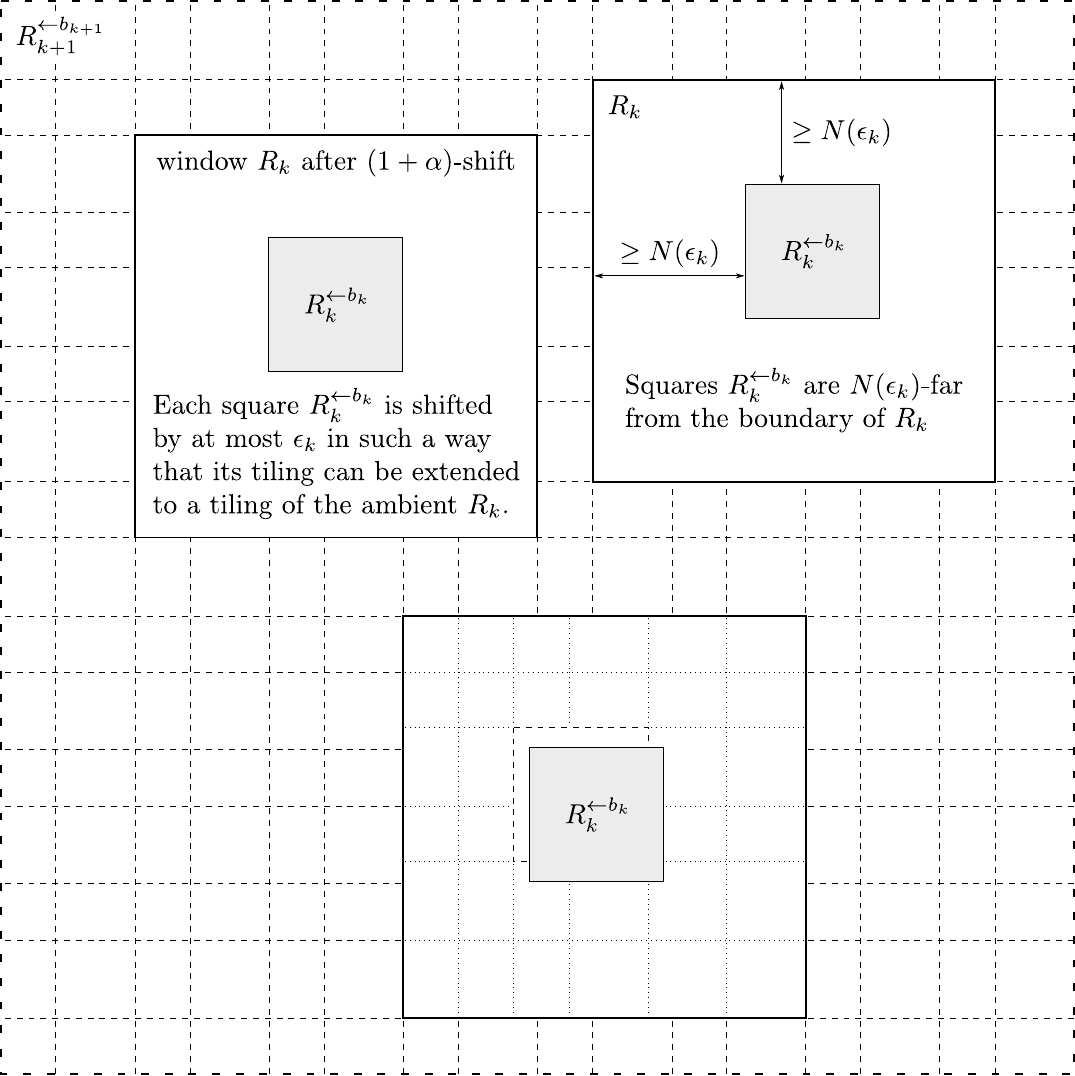}
    \caption{Extending tiling from \( \mathcal{C}_{n} + R_{n}^{\la b_{n}} \) to \( \mathcal{C}_{n+1}
      + R_{n+1}^{\la b_{n+1}}\).}
    \label{fig:extending-tiling-general-case}
  \end{figure}

  Now we are going to move each \( R_{k}^{\la b_{k}} \) region by a vector of
  \( \ell^{\infty} \)-norm at most \( \epsilon_{k} \); according to the algorithm described at the
  beginning of this sketch, for each \( c_{i} \) one can find \( \vec{v}_{i} \in \mathbb{R}^{d} \),
  \( \|\vec{v}_{i}\|_{\infty} < \epsilon_{k} \), such that the tiling of
  \( c_{i} + \vec{v}_{i} + R_{k}^{\la b_{k}} \) can be extended to the tiling of the window
  \( R_{k} \) around \( c_{i} \).  On Figure \ref{fig:extending-tiling-general-case} the bottom
  \( R_{k} \) exhibits this process.  Finally, the gaps between \( R_{k} \) are filled by
  the canonical tiling, which extends the tiling of each \( R_{k} \) to a tiling of
  \( R_{k+1}^{\la b_{k+1}} \).  To summarize, when extending the tiling of
  \( \mathcal{C}_{k} + R_{k}^{\la b_{k}} \) to a tiling of
  \( \mathcal{C}_{k+1} + R_{k+1}^{\la b_{k+1}} \), we shift each \( c \in \mathcal{C}_{k} \)
  (together with all the tiles in \( c + R_{k}^{\la b_{k}} \)) by no more than \( \epsilon_{k} \) in
  each coordinate direction.
  This describes the step of induction.

  While strictly speaking the tiling of \( \mathcal{C}_{k+1} + R_{k+1}^{\la b_{k+1}} \) extends the
  tiling of \( \mathcal{C}_{k} + R_{k}^{\la b_{k}} \) only up to an \( \epsilon_{k} \)-shift, using
  \( \sum_{k} \epsilon_{k} < \infty \) we may naturally define the limit tiling of
  \( \bigcup_{k} \bigl( \mathcal{C}_{k} + R_{k}^{\la b_{k}} \bigr) = Z \).  This results in a
  construction of a tiling of \( Z \) by tiles of the form \( \widetilde{R}_{\vec{a}} \),
  \( \vec{a} \in \{1,\alpha\}^{d} \), like in Figure \ref{fig:regular-tiling}.

  So, let \( \widetilde{\mathcal{C}} \) denote the set of centers of all the tiles in \( Z \), and let
  \[ \widetilde{\mathcal{C}} = \bigsqcup_{\vec{a} \in
    \{1,\alpha\}^{d}}\widetilde{\mathcal{C}}_{\vec{a}}\]
  be the decomposition of tiles into the \( 2^{d} \) types according to lengths of their sides.
  
  It remains to explain how the maps
  \( \theta_{\vec{a}} : \widetilde{\mathcal{C}}_{\one} \to \widetilde{\mathcal{C}}_{\vec{a}} \) are
  constructed.  It is immediate from the construction that for any \( c \in \mathcal{C}_{k} \) the
  tiling of \( c + R_{k}^{\la b_{k}} \) has equally many tiles of each type.  For a given
  \( c \in \mathcal{C}_{k+1} \), if \( \theta_{\vec{a}} \) is defined on each
  \( c_{i} + R_{k}^{\la b_{k}} \),
  \( c_{i} \in \mathcal{C}_{k} \cap \bigl( c + R_{k+1}^{\la b_{k+1}} \bigr) \), then in
  \( (c + R_{k+1}^{\la b_{k+1}}) \setminus (\mathcal{C}_{k} + R_{k}) \) we have the same number of
  tiles of each type, and therefore the map \( \theta_{\vec{a}} \) can be extended to a matching
  between \( \one \)-tiles and \( \vec{a} \)-tiles on \( c + R_{k+1}^{\la b_{k+1}} \) in a Borel
  way.  In the limit \( \theta_{\vec{a}} \) is a matching from \( \widetilde{\mathcal{C}}_{\one} \)
  onto \( \widetilde{\mathcal{C}}_{\vec{a}} \), as desired.
\end{proof}

We conclude this section by showing how the above theorem gives rise to a LOE between invariant
subsets of uniformly full measure.

\begin{theorem}
  \label{thm:LOE-uniformly-full}
  Let \( \mathbb{R}^{d} \acts X \) and \( \mathbb{R}^{d} \acts Y \) be a pair of free Borel flows
  having the same cardinality of the sets of pie measures.  There exist invariant Borel
  subsets \( Z_{X} \subseteq X \) and \( Z_{Y} \subseteq Y \) of uniformly full measure and a LOE \(
  \phi : Z_{X} \to Z_{Y} \) between restrictions of the flows.
\end{theorem}

\begin{proof}
  If the flows admit no pie measures, the statement is vacuously true, since one may take
  \( Z_{X} = \es = Z_{Y} \).  We therefore assume that flows have invariant measures.  One starts
  with Theorem \ref{thm:Rudolphs-tiling} to find \( Z_{X} \subseteq X \) and \( Z_{Y} \subseteq Y \)
  of uniformly full measure together with regular tilings associated with cross sections
  \( \mathcal{C}_{X} \subseteq Z_{X} \) and \( \mathcal{C}_{Y} \subseteq Z_{Y} \).  Notice that
  \( |\mathcal{E}(X)| = |\mathcal{E}(Z_{X})| \) and \( |\mathcal{E}(Y)| = |\mathcal{E}(Z_{Y})| \),
  because \( X \setminus Z_{X} \) and \( Y \setminus Z_{Y} \) have measure zero with respect to all
  invariant measures.

  We would like to apply DJK classification of hyperfinite equivalence relations to cross sections
  \( \mathcal{C}_{X, \one} \) and \( \mathcal{C}_{Y, \one} \) and to find a Borel isomorphism
  between induced equivalence relations \( \phi : \mathcal{C}_{X,\one} \to \mathcal{C}_{Y,\one} \).
  For this we need to check that the equivalence relations on these cross sections possess the same
  number of pie measures.  While Proposition \ref{prop:bijection-pie-measures} shows that
  \( |\mathcal{E}(Z_{X})| = |\mathcal{E}(\mathcal{C}_{X})|\), it is not immediately clear whether
  this proposition can be applied to the sub cross section \( \mathcal{C}_{X,\one} \), as it is not
  evident from the construction of Theorem \ref{thm:Rudolphs-tiling} whether
  \( \mathcal{C}_{X, \one} \) is cocompact in \( Z_{X} \).  While it is possible to modify the
  argument in Theorem \ref{thm:Rudolphs-tiling} to ensure cocompactness of all
  \( \mathcal{C}_{\vec{a}} \), we may show
  \( |\mathcal{E}(\mathcal{C}_{X})| = |\mathcal{E}(\mathcal{C}_{X, \one})| \) instead as follows.
  Since all the matchings
  \( \theta^{X}_{\vec{a}} : \mathcal{C}_{X, \one} \to \mathcal{C}_{X, \vec{a}} \) preserve the
  equivalence relation \( E_{X} \), they also preserve all the invariant measures on
  \( \mathcal{C}_{X} \).  So, if \( \mu \) is an invariant probability measure on
  \( \mathcal{C}_{X} \), then \( \mu(\mathcal{C}_{X}) = 2^{d} \mu(\mathcal{C}_{X, \one}) \) and
  \( 2^{d}\mu|_{\mathcal{C}_{X, \one}} \) is an invariant probability measure on
  \( \mathcal{C}_{X, \one} \).  On the other hand, if \( \nu \) is an invariant measure on
  \( \mathcal{C}_{X, \one} \), then
  \[ \tilde{\nu} = 2^{-d}\mkern-18mu\sum_{\vec{a} \in \{1, \alpha\}^{d}}
  \mkern-14mu(\theta^{X}_{\vec{a}})_{*} \nu, \quad \textrm{where } \theta^{X}_{\one} = \mathrm{id}, \]
  is easily seen to be an invariant measure on \( \mathcal{C}_{X} \).  These maps,
  \( \mu \mapsto 2^{d}\mu|_{\mathcal{C}_{X, \one}} \) and \( \nu \mapsto \tilde{\nu} \), are inverses
  of each other and are bijections between \( \mathcal{E}(\mathcal{C}_{X}) \) and
  \( \mathcal{E}(\mathcal{C}_{X, \one}) \).  We conclude that
  \( |\mathcal{E}(\mathcal{C}_{X})| = |\mathcal{E}(\mathcal{C}_{X, \one})| \) and therefore 
  \[ |\mathcal{E}(\mathcal{C}_{X, \one})| = |\mathcal{E}(Z_{X})| = |\mathcal{E}(Z_{Y})| =
  |\mathcal{E}(\mathcal{C}_{Y, \one})|.\]

  This allows us to apply the DJK classification and get an isomorphism between the restrictions of
  orbit equivalence relations \( \phi : \mathcal{C}_{X, \one} \to \mathcal{C}_{Y, \one} \) (recall
  that \( E_{\mathcal{C}_{X}} \) and \( E_{\mathcal{C}_{Y}} \) are necessarily hyperfinite by
  \cite[Theorem 1.16]{jackson_countable_2002}).  The maps \( \theta_{\vec{a}}^{X} \) and
  \( \theta_{\vec{a}}^{Y} \) make it easy to extend \( \phi \) to an isomorphism
  \( \phi : \mathcal{C}_{X} \to \mathcal{C}_{Y} \) by setting
  \[ \phi \circ \theta_{\vec{a}}^{X}(c) = \theta^{Y}_{\vec{a}} \circ \phi(c) \quad \textrm{for each
    \( \vec{a} \in \{1,\alpha\}^{d} \) and all \( c \in \mathcal{C}_{X, \one} \)}. \]

  Finally, we may extend \( \phi \) linearly to a LOE \( \phi: Z_{X} \to Z_{Y} \).  More formally,
  for any \( x \in Z_{X} \) there exist unique \( \vec{a} \in \{1,\alpha\}^{d} \) and \( c \in
  \mathcal{C}_{X, \vec{a}} \) such that \( x \in c + \widetilde{R}_{\vec{a}} \).  Let \( \vec{v} \in
  \widetilde{R}_{\vec{a}}\) be such that \( c + \vec{v} = x \).  The point \( \phi(x) \) is defined by
  \[ \phi(x) = \phi(x - \vec{v}) + \vec{v}. \]
  In other words, \( \phi \) maps \( c + \widetilde{R}_{\vec{a}} \) onto
  \( \phi(c) + \widetilde{R}_{\vec{a}} \) in a linear fashion.  It is, of course, the crucial
  property of our construction that \( c \in \mathcal{C}_{X, \vec{a}} \) if and only if
  \( \phi(c) \in \mathcal{C}_{Y,\vec{a}} \).
\end{proof}

\section{Lebesgue Orbit Equivalences between compressible flows}
\label{sec:compr-flows-proof}

In this section we deal with flows that have no invariant probability measures.  This negative
condition has a positive reformulation discovered by M.~G.~Nadkarni \cite{nadkarni_existence_1990}.
\begin{theorem-nn}[Nadkarni]
  A hyperfinite\footnote{As proved in \cite[Theorem 4.3.1]{becker_descriptive_1996}, the assumption of
    hyperfiniteness may be dropped; the theorem is true for all countable Borel equivalence
    relations.} Borel equivalence relation has no invariant probability measures if and only if it
  is compressible.
\end{theorem-nn}

There is a number of equivalent reformulations of compressibility, we shall adopt the following one.
A countable Borel Equivalence relation \( E \) on a standard Borel space \( \mathcal{C} \) is
\emph{compressible} if there exist injective homomorphisms
\( \tau_{n} : \mathcal{C} \to \mathcal{C} \), \( n \in \mathbb{N} \), with disjoint images:
\( \tau_{m}(\mathcal{C}) \cap \tau_{n}(\mathcal{C}) = \es \) when \( m \ne n \) and
\[ c E \tau_{n}(c) \quad \textrm{for all } c \in \mathcal{C} \textrm{ and all } n \in
\mathbb{N}.  \]
This homomorphisms will allow us to run a back-and-forth construction of LOE map between
compressible flows.

\begin{theorem}
  \label{thm:compressible-LOE}
  If free non smooth flows \( \mathbb{R}^{d} \acts X \) and \( \mathbb{R}^{d} \acts Y \)
  admit no invariant probability measures, then these flows are Lebesgue Orbit Equivalent.
\end{theorem}

\begin{proof}
  We begin by applying Theorem \ref{thm:existence-of-rectangular-tilings-restr} to pick rectangular
  tilings \( \mathscr{R}^{X} \) and \( \mathscr{R}^{Y} \) of \( X \) and \( Y \).  We agreed earlier
  to pick the center of each tile as its representing point, but it is convenient to deviated from
  this convention here and let \( \mathcal{C}_{X} \) and \( \mathcal{C}_{Y} \) to consists of
  ``bottom left'' corners of the domains \( \mathscr{R}^{X}_{c} \) and \( \mathscr{R}^{Y}_{c} \).
  For any \( c \in \mathcal{C}_{X} \cup \mathcal{C}_{Y} \) we therefore have a rectangle
  \( R_{c} \subseteq \mathbb{R}^{d} \) of the form \( \prod_{i=1}^{d}\bigl[0, \zeta_{i}(c)\bigr) \)
  such that
  \[ \mathscr{R}^{X} = \bigl\{(x,c) \in X \times \mathcal{C}_{X} : x \in c + R_{c}\bigr\}\quad
  \textrm{ and }\quad \mathscr{R}^{Y} = \bigl\{(x,c) \in Y \times \mathcal{C}_{Y} : x \in c +
  R_{c}\bigr\}. \]
  According to Theorem \ref{thm:existence-of-rectangular-tilings-restr}, we may assume that
  \( \zeta_{i}(c) \in [4,5] \) for all \( c \in \mathcal{C}_{X} \cup \mathcal{C}_{Y} \) and all
  \( i \le d \).
  
  These cross sections are cocompact, and by the results from Section
  \ref{sec:invar-meas-cross-phase}, equivalence relations \(E_{{\mathcal{C}}_{X}} \) and
  \( E_{\mathcal{C}_{Y}} \) have no invariant probability measures.  By \cite[Theorem
  1.16]{jackson_countable_2002}, these relations are also hyperfinite, and therefore by the DJK
  classification, there is a Borel isomorphism \( \phi : \mathcal{C}_{X} \to \mathcal{C}_{Y} \)
  between \( E_{\mathcal{C}_{X}} \) and \( E_{\mathcal{C}_{Y}} \).

  In what follows we extend \( \phi \) to a LOE  \( \phi: X \to Y \) between \( E_{X} \) and \( E_{Y} \).

  For \( k \ge 0 \), let \( B^{k} = \prod_{i \le d} [0, 2^{-k}) \) denote a semi-open
  \( d \)-dimensional square with side \( 2^{-k} \), and let us fix for a moment a single tile
  \( R_{c} \) for some \( c \in \mathcal{C}_{X} \).  We describe a process of covering a portion of
  \( R_{c} \) by copies of \( B^{k} \).  Let \( n_{i,k}(c) = n_{i,k} \) be the smallest
  integer such that
  \begin{equation}
    \label{eq:5}
    \frac{2^{-k} \cdot n_{i,k}}{\zeta_{i}(c)} \in \bigl( 1-2^{-k-1}, 1  \bigr).
  \end{equation}

  Here is a verbose explanation of this parameter.  Consider the interval
  \( \bigl[0, \zeta_{i}(c)\bigr) \), and let us start tiling it with intervals of length
  \( 2^{-k} \) beginning from the left endpoint.  The integer \( n_{i,k} \) is the smallest integer
  such that if we put \( n_{i,k} \) many intervals \( [0,2^{-k}) \) into
  \( \bigl[0, \zeta_{i}(c)\bigr) \), then the proportion of \( \bigl[0, \zeta_{i}(c)\bigr) \) that
  that is not covered is less than \( 2^{-k-1} \).  Since in our situation \( \zeta_{i}(c) \) is
  always between \( 4 \) and \( 5 \), \( n_{i,0}(c) = 3 \) for all \( i \) and all \( c \), but for
  \( k \ge 1 \) the parameter will start to vary.

  Note that since \( n_{i,k} \) is defined to be the smallest integer satisfying \eqref{eq:5}, one
  has in fact
  \begin{equation}
    \label{eq:6}
    \frac{2^{-k} \cdot n_{i,k}}{\zeta_{i}(c)} \in \bigl( 1-2^{-k-1}, 1-2^{-k-2}  \bigr] \quad
    \textrm{for all \( i \le d \) and \( k \ge 0 \)}.
  \end{equation}

  \addpicture{r}{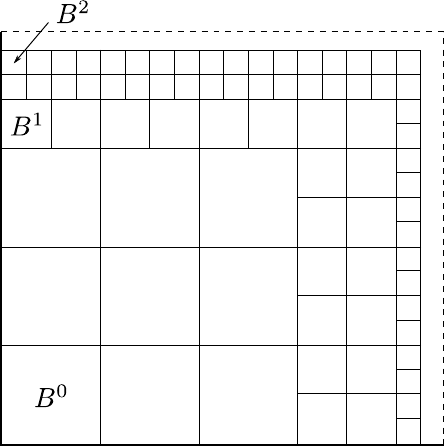}{45.2mm}{\vspace*{6mm}}{fig:tiles-decomposition}
  We now partially cover \( R_{c} \) with copies of \( B^{k} \) by putting \( n_{i,k} \) many
  rectangles \( B^{k} \) in the \( i^{\textrm{th}} \) direction of \( R_{c} \) starting from the
  ``bottom left'' corner.  

  Figure \ref{fig:tiles-decomposition} illustrates the levels \( k = 0,1\), and \( 2 \) of this
  process.  On that figure, \( n_{1,0} = n_{2,0} = 3 \), and we therefore have \( 9 \) squares
  \( B^{0} \) \textemdash{} three in each row and each column.  At the level \( k = 1 \),
  \( n_{1,1} = 8 \) and \( n_{1,1} = 7 \) resulting in \( 56 \) copies of \( B^{1} \).  Note that in
  Figure \ref{fig:tiles-decomposition} blocks \( B^{1} \) which refine those of \( B^{0} \) are not
  shown; only the blocks which cover parts of \( R_{c} \) uncovered by \( B^{0} \) are depicted.
  Finally, for \( k=2 \) we have \( n_{1,2} = 17 \) and \( n_{2,2} = 16 \) and blocks
  \( B^{2} \) cover even more of \( R_{c} \).

  This partial covers can be constructed in a Borel fashion for all points
  \( c \) in \( \mathcal{C}_{X} \cup \mathcal{C}_{Y} \), which results in chains of Borel cross sections
  \[\hspace*{-5cm} \mathcal{C}_{X} \subseteq \mathcal{C}_{X}^{0} \subseteq \mathcal{C}_{X}^{1}
  \subseteq \cdots,
  \quad \textrm{and} \quad \mathcal{C}_{Y} \subseteq \mathcal{C}_{Y}^{0} \subseteq
  \mathcal{C}_{Y}^{1} \subseteq \cdots, \]
  where \( \mathcal{C}_{X}^{k} \) consists of ``bottom left'' endpoints of blocks \( B^{k} \), that
  satisfy the following properties:

  \vspace*{0.01mm}
  \begin{enumerate}[(a)]
  \item \( \mathcal{C}_{X}^{k} + B^{k} \subseteq \mathcal{C}_{X}^{k+1} + B^{k+1}  \) \textemdash{} the next
    level of blocks covers at least as much as the previous.
  \item\label{item:next-cover-has-more}
    \( (c + R_{c}) \cap \bigl( \mathcal{C}_{X}^{k+1} \setminus (\mathcal{C}_{X}^{k} + B^{k}) \bigr) \ne
    \es \) for all \( c \in \mathcal{C}_{X} \) \textemdash{} within every tile \( R_{c} \) there is
    always a point from \( \mathcal{C}_{X}^{k+1} \) which has not been covered by any
    \( B^{k} \) block.  In other words, within every tile blocks \( B^{k+1} \) cover \emph{strictly
      more} than \( B^{k} \) blocks.
  \item\label{item:covering-each-tile} \( X = \bigcup_{k}(\mathcal{C}_{X}^{k} + B^{k}) \)
    \textemdash{} every point in \( X \) is covered from some level on.
  \end{enumerate}
  Of course, similar properties hold for \( \mathcal{C}_{Y} \) instead of \( \mathcal{C}_{X} \).
  
  We are now ready to run a back-and-forth extension of
  \( \phi : \mathcal{C}_{X} \to \mathcal{C}_{Y} \) beginning with the step \( k = 0 \).  Since we
  have chosen our tiles in such a way that \( n_{i,0} = 3 \) for all \( i \), each tile \( R_{c} \)
  has \( 3^{d} \) blocks \( B^{0} \).

  We therefore may extend \( \phi \) first to a Borel isomorphism \( \phi : \mathcal{C}_{X}^{0} \to
  \mathcal{C}_{Y}^{0} \) between \( E_{\mathcal{C}_{X}^{0}} \) and \( E_{\mathcal{C}_{Y}^{0}} \) by
  matching \( (c + R_{c}) \cap \mathcal{C}_{X}^{0} \) with point in \( (\phi(c) + R_{\phi(c)}) \cap
  \mathcal{C}_{Y}^{0} \) for all \( c \in \mathcal{C}_{X} \), and then extend \( \phi :
  \mathcal{C}_{X}^{0} + B^{0} \to \mathcal{C}_{Y}^{0} + B^{0}\) linearly on each block \( B^{0} \).
  The map \( \phi \) defined this way preserves Lebesgue measure within orbits on its domain.  Since
  \( n_{i,0}(c) = 3 \) for all \( c \in \mathcal{C}_{X} \) and all \( i \le d \), there is no need
  for the ``back'' part of the argument and we proceed to the next step of the construction.

  At the level \( k=1 \) we would like to extend \( \phi \) which is currently defined on
  \( \mathcal{C}_{X}^{0} + B^{0} \) to a map
  \( \phi : \mathcal{C}_{X}^{1} + B^{1} \to \mathcal{C}_{Y}^{1} + B^{1} \).  A naive approach
  would be to take \( c \in \mathcal{C}_{X} \) and to try to map injectively the elements
  \( \mathcal{C}_{X}^{1} \setminus (\mathcal{C}_{X}^{0} + B^{0}) \) from \( c + R_{c} \) to
  corresponding elements from \( \phi(c) + R_{\phi(c)} \).  This approach may fail, since there may
  be more elements in the domain, than in the range.  For instance, in example on Figure
  \ref{fig:forth-step} we have \( n_{1,1}(c) = 8 \) and \( n_{2,1}(c) = 8 \), while in the
  image \( R_{\phi(c)} \) we may have \( n_{1,1}(\phi(c)) = 7 \) and \( n_{2,1}(\phi(c)) = 7 \), and
  so there will be \( 28 \) blocks \( B^{1} \) in \( c + R_{c} \) not in the domain of \( \phi \) at
  the current stage, and only \( 13 \) blocks \( B^{1} \) in \( \phi(c) + R_{\phi(c)} \) not in the
  range of \( \phi \).

  \begin{figure}[ht]
    \centering
    \begin{overpic}[tics=10]{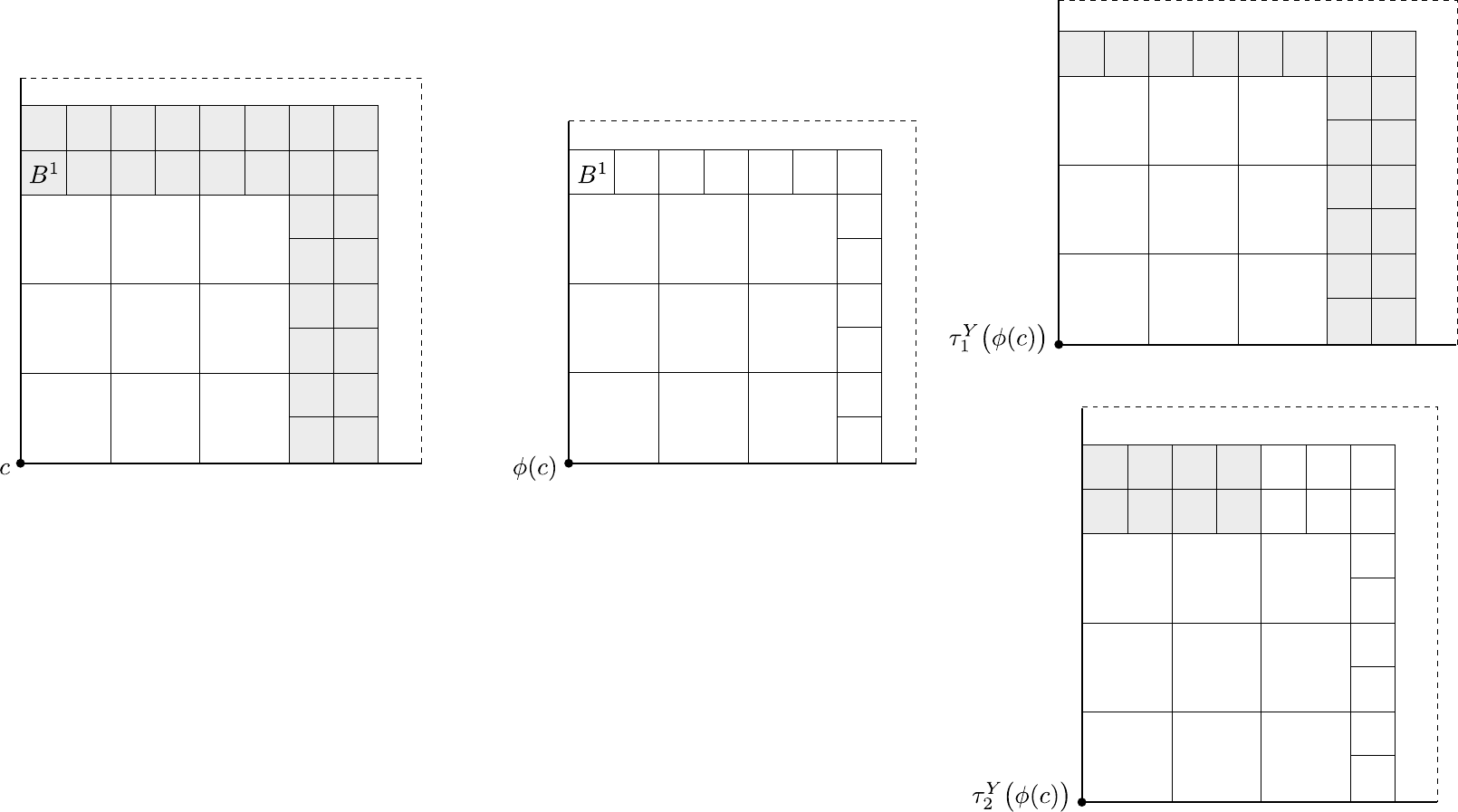}
      \put(39,49){
        \begin{minipage}[b]{4cm}
          Fewer \( B^1 \) blocks in \( R_{\phi(c)} \)\\ than in \( R_c \).
        \end{minipage}}
      \put (2,4){
      \begin{minipage}[b]{10cm}
        While there are not enough \( B^{1} \) blocks available in \( R_{\phi(c)} \), there are more
        \( B^{1} \) blocks not in the range of \( \phi \) in \( R_{\tau_1^Y(\phi(c))} \) and
        \( R_{\tau_2^Y(\phi(c))} \) combined than in \( R_{c} \).  Gray \( B^{1} \) blocks in \( R_{c} \)
        are mapped linearly onto gray \( B^{1} \) blocks in \( R_{\tau_1^Y(\phi(c))} \) and
        \( R_{\tau_2^Y(\phi(c))} \).  Note that some of the \( B^{1} \) blocks in
        \( R_{\tau_2^Y(\phi(c))} \) remain available.  We shall take care of them during the
        ``back'' stage of the construction.
      \end{minipage}}
    \end{overpic}
    \caption{Forth step extension}
    \label{fig:forth-step}
  \end{figure}

  We overcome this obstacle by using the maps
  \( \tau_{n}^{Y} : \mathcal{C}_{Y} \to \mathcal{C}_{Y} \) which witness compressibility.  By item
  \eqref{item:next-cover-has-more}, for any
  \( \tilde{c} \in \mathcal{C}_{Y} \) there exists at least one \( B^{1} \) block in
  \( \tilde{c} + R_{\tilde{c}} \) which is not in the range of \( \phi \).  On the other hand,
  there is a uniform upper bound on the number of \( B^{1} \) blocks in any \( c + R_{c} \) tile,
  \( c \in \mathcal{C}_{X} \).  Therefore for \( N_{1} \) sufficiently large the total number of
  \( B^{1} \) blocks in tiles
  \[ \tau^{Y}_{1}(\phi(c)) + R_{\tau^{Y}_{1}(\phi(c))}, \ldots, \tau^{Y}_{N_{1}}(\phi(c)) +
  R_{\tau^{Y}_{N_{1}}(\phi(c))} \]
  which are not in the range of \( \phi \) exceeds the number of \( B^{1} \) blocks in
  \( c + R_{c} \).  We can thus extend \( \phi \) to an injective function
  \( \phi : \mathcal{C}_{X}^{1} \to \mathcal{C}_{Y}^{1} \) by mapping for each
  \( c \in \mathcal{C}_{X} \)
  \[ \textrm{points in } (c + R_{c}) \cap \bigl(\mathcal{C}_{X}^{1} \setminus (\mathcal{C}_{X}^{0} +
  B^{0})\bigr) \textrm{ to points in } \bigcup_{j=1}^{N_{1}}\Bigl(\bigl(\tau^{Y}_{j}(\phi(c)) +
  R_{\tau^{Y}_{j}(\phi(c))}\bigr) \cap \bigl(\mathcal{C}_{Y}^{1} \setminus (\mathcal{C}_{Y}^{0} + B^{0})
  \bigr) \Bigr), \]
  and then extend it linearly to
  \( \phi : \mathcal{C}_{X}^{1} + B^{1} \to \mathcal{C}_{Y}^{1} + B^{1} \).  This finishes the
  ``forth'' part of our back-and-forth argument.

  The ``back'' part of the argument is similar.  At this moment \( \phi \) is defined on all
  \( B^{1} \) blocks in \( X \), while the range of \( \phi \) consists of \emph{all} \( B^{0} \)
  blocks and \emph{some} \( B^{1} \) blocks in \( Y \).  In other words, \( \phi^{-1} \) is defined
  on some of the \( B^{1} \) blocks of \( Y \), and we would like to extend \( \phi^{-1} \) to all
  blocks \( B^{1} \) in \( Y \).  We shall map \( B^{1} \) blocks of \( Y \) which are not yet in
  the domain of \( \phi^{-1} \) onto \( B^{2} \) blocks in \( X \).  Each \( B^{1} \) block will be
  mapped (in a measure preserving way) onto \( 2^{d} \) blocks \( B^{2} \).  Since there is a
  uniform bound on the number of \( B^{1} \) blocks in a tile \( \tilde{c} + R_{\tilde{c}} \),
  \( \tilde{c} \in \mathcal{C}_{Y} \), and since each tile \( c + R_{c} \),
  \( c \in \mathcal{C}_{X} \), contains at least one \( B^{2} \) block not yet in the domain of
  \( \phi \)(because of item \eqref{item:next-cover-has-more} and because the domain of \( \phi \)
  currently consists of \( B^{1} \) rectangles only), there exists a sufficiently large \( M_{1} \)
  such that for any \( \tilde{c} \in \mathcal{C}_{Y} \) the number of available \( B^{2} \) blocks
  in the tiles
  \[ \tau_{1}^{X}\bigl( \phi^{-1}(\tilde{c}) \bigr) + R_{\tau_{1}^{X}( \phi^{-1}(\tilde{c}) )}, \ldots,
  \tau_{M_{2}}^{X}\bigl( \phi^{-1}(\tilde{c}) \bigr) + R_{\tau_{M_{1}}^{X}( \phi^{-1}(\tilde{c}) )} \]
  exceeds 
  \[ 2^{d} \cdot \bigl| (\tilde{c} + R_{\tilde{c}}) \cap \mathcal{C}^{1}_{Y} \bigr|. \]
  We may therefore extend \( \phi \) in such a way that \( \phi^{-1} \) is defined on all of
  \( \mathcal{C}_{Y}^{1} + B^{1} \) and \( \phi \) preserves the Lebesgue measure on its domain.
  This ends the ``back'' part.

  The construction continues in the same fashion.  The map \( \phi \) is now define on \emph{some}
  \( B^{2} \) blocks in \( X \) and we extend it to all of \( \mathcal{C}_{X}^{2} + B^{2} \) in such
  a way that the image of a \( B^{2} \) block in \( X \) is a \( B^{2} \) block in \( Y \).  In
  general, \( \phi \) will satisfy
  \[ \phi(\mathcal{C}_{X}^{k} + B^{k}) \subseteq \mathcal{C}_{Y}^{k} + B^{k} \quad \textrm{and}
  \quad \phi^{-1}(\mathcal{C}_{Y}^{k} + B^{k}) \subseteq \mathcal{C}_{X}^{k+1} + B^{k+1}.\]
  From item \eqref{item:covering-each-tile}, it is immediate that in the limit \( \phi \) is a Borel
  isomorphism between \( X \) and \( Y \), and the construction ensures that \( \phi : X \to Y \) is
  a Lebesgue Orbit Equivalence.
\end{proof}

\section{Proof of the main theorem}
\label{sec:proof-main-theorem}

\begin{theorem}
  \label{thm:loe-Rd-flows}
  Let \( \mathbb{R}^{d} \acts X \) and \( \mathbb{R}^{d} \acts Y \) be a pair of free non smooth
  flows.  These flows are LOE if and only if \( |\mathcal{E}(X)| = |\mathcal{E}(Y)| \).
\end{theorem}
\begin{proof}
  Necessity was proved in Theorem \ref{thm:wHOE-bijection-pie}.  We prove sufficiency.  The
  combination of Theorem \ref{thm:LOE-uniformly-full} and Theorem \ref{thm:compressible-LOE} almost
  works: we may select invariant \( Z_{X} \subseteq X \) and \( Z_{Y} \subseteq Y \) with a LOE
  \( \phi : Z_{X} \to Z_{Y} \) and reduce the problem to finding a LOE between flows
  \( \mathbb{R}^{d} \acts X \setminus Z_{X} \) and \( \mathbb{R}^{d} \acts Y \setminus Z_{Y} \)
  which have no invariant measures.  The only difficulty in applying Theorem
  \ref{thm:compressible-LOE} to the latter is that \( X \setminus Z_{X} \) or
  \( Y \setminus Z_{Y} \) may be smooth.  Fortunately, this obstacle is easy to overcome.

  Let \( X_{0} \subseteq X \) and \( Y_{0} \subseteq Y \) be invariant subsets such that the
  restrictions \( \mathbb{R}^{d} \acts X_{0} \) and \( \mathbb{R}^{d} \acts Y_{0} \) are not smooth
  but admit no pie measures.  Such subsets can be selected as follows.  Pick a cocompact cross
  section \( \mathcal{C} \subseteq X \) and select an invariant Borel subset
  \( \mathcal{C}_{0} \subseteq \mathcal{C} \) such that the equivalence relation
  \( E_{\mathcal{C}_{0}} \) is compressible, but not smooth (see, for instance, \cite[Corollary
  7.2]{dougherty_structure_1994}).  Set \( X_{0} \) to be the saturation of
  \( \mathcal{C}_{0} \), \( X_{0} = \mathcal{C}_{0} + \mathbb{R}^{d} \).  The set \( Y_{0} \subseteq
  Y\) can be selected in a similar way.

  Having picked such \( X_{0} \) and \( Y_{0} \), let \( X' = X \setminus X_{0} \) and
  \( Y' = Y \setminus Y_{0} \).  Apply now Theorem \ref{thm:LOE-uniformly-full} to flows
  \( \mathbb{R}^{d} \acts X' \) and \( \mathbb{R}^{d} \acts Y' \).  As an output we get a LOE
  \( \phi : Z_{X'} \to Z_{Y'} \) between subsets of uniformly full measure.  Now consider the flows
  restricted to the complements:
  \[ \mathbb{R}^{d} \acts X_{0} \cup \bigl( X' \setminus Z_{X'} \bigr) \textrm{ and } \mathbb{R}^{d}
  \acts Y_{0} \cup \bigl( Y' \setminus Z_{Y'} \bigr). \]
  These have no pie measures and are necessarily non smooth, whence we may apply Theorem
  \ref{thm:compressible-LOE} to extend \( \phi \) to a LOE between \( \mathbb{R}^{d} \acts X \) and
  \( \mathbb{R}^{d} \acts Y \).
\end{proof}

\bibliographystyle{alpha}
\bibliography{refs}

\end{document}